\documentclass[a4paper,11pt]{article}
\usepackage{MyStyle}
\usepackage{aromatic_forms}

\usepackage{hyperref}

\usepackage{algorithm}
\floatname{algorithm}{}
\usepackage{algorithmic}

\def\tOmega{\widetilde{\Omega}}
\def\im{\mathop\mathrm{Im}}

\newcommand{\starsymbol}{\diamond}
\newcommand{\Eulerop}{\EE}

\usepackage{circledsteps}

\allowdisplaybreaks

\newtheorem{theorem}{Theorem}[section]
\newtheorem{definition}[theorem]{Definition}
\newtheorem*{definition*}{Definition}

\newtheorem{proposition}[theorem]{Proposition}
\newtheorem{lemma}[theorem]{Lemma}
\newtheorem{remark}[theorem]{Remark}
\newtheorem*{remark*}{Remark}

\newtheorem*{remarks*}{Remarks}
\newtheorem{corollary}[theorem]{Corollary}

\newtheorem*{notation*}{Notation}

\newtheorem*{ex*}{Example}

\newtheorem*{exs*}{Examples}

\newtheorem*{app*}{Application}
\newtheorem{conjecture*}{Conjecture}

\def\ts{\thinspace}

\title{
The aromatic bicomplex for the description of divergence-free aromatic forms and volume-preserving integrators
}

\author{
Adrien Laurent\textsuperscript{1}, Robert I.\ts McLachlan\textsuperscript{2}, Hans Z.\ts Munthe-Kaas\textsuperscript{1}, and Olivier Verdier\textsuperscript{3}
}

\begin{document}
\footnotetext[1]{
Department of Mathematics, University of Bergen, Bergen, Norway.\\
adrien.laurent@uib.no. hans.munthe-kaas@uib.no.}
\footnotetext[2]{
Institute of Fundamental Sciences, Massey University, Palmerston North, New Zealand.\\
r.mclachlan@massey.ac.nz.}
\footnotetext[3]{
Department of Computing, Electrical Engineering and Mathematical Sciences,Western Norway University of Applied Sciences, Bergen, Norway.\\
olivier.verdier@hvl.no.}

\maketitle

\begin{abstract}
Aromatic B-series were introduced as an extension of standard Butcher-series for the study of volume-preserving integrators.
It was proven with their help that the only volume-preserving B-series method is the exact flow of the differential equation.
The question was raised whether there exists a volume-preserving integrator that can be expanded as an aromatic B-series.
In this work, we introduce a new algebraic tool, called the aromatic bicomplex, similar to the variational bicomplex in variational calculus.
We prove the exactness of this bicomplex and use it to describe explicitly the key object in the study of volume-preserving integrators: the aromatic forms of vanishing divergence.
The analysis provides us with a handful of new tools to study aromatic B-series, gives insights on the process of integration by parts of trees, and allows to describe explicitly the aromatic B-series of a volume-preserving integrator. In particular, we conclude that an aromatic Runge-Kutta method cannot preserve volume.

\smallskip

\smallskip

\end{abstract}


\section{Introduction}

Let~$f\colon \R^d \rightarrow \R^d$ be a smooth Lipschitz vector field and let~$y\colon [0,T]\rightarrow \R^d$ be the solution to the ordinary differential equation
\begin{equation}
\label{equation:ODE}
y'(t)=f(y(t)), \quad t\in ]0,T[, \quad y(0)=y_0.
\end{equation}
If the vector field~$f$ is divergence-free, that is, if~$\Div(f)=0$, it is known that the solution of the ODE~\eqref{equation:ODE} is volume-preserving, that is, for any measurable set of initial conditions~$D$ with respect to the Lebesgue measure~$\lambda$, for any~$t>0$, the flow~$\varphi_t$ of the ODE~\eqref{equation:ODE} satisfies~$\lambda(\varphi_t(D))=\lambda(D)$.
Divergence-free vector fields appear in a variety of concrete dynamical systems, for instance in fluid dynamics, meteorology or molecular dynamics.
To integrate such systems, it is fundamental to use numerical integrators that also preserve the volume.
We discretize the time interval~$[0,T]$ into~$N+1$ equidistant steps~$t_n=nh$, with~$h$ the time stepsize, and we choose a one-step integrator
\begin{equation}
\label{equation:one-step_integrator}
y_{n+1}=\Phi(y_n,h).
\end{equation}
For large classes of integrators, the standard methodology to study volume-preservation uses backward error analysis (see, for instance, the textbook~\cite{Hairer06gni}).
The integrator~\eqref{equation:one-step_integrator} can be interpreted formally as the exact solution of a modified ODE
\begin{equation}
\label{equation:modified_ODE}
\widetilde{y}'(t)=\widetilde{f}(\widetilde{y}(t)),
\end{equation}
where the modified vector field~$\widetilde{f}$ typically depends on~$f$ and its derivatives. Then, the integrator is volume-preserving if and only if~$\Div(\widetilde{f})=0$.

There is a considerable literature on volume-preserving methods, with many applications for solving a variety of dynamical systems.
The existing volume-preserving integrators rely either on a specific form of the vector field~$f$~\cite{Hairer06gni,Chartier07pfi}, on splitting methods~\cite{Kang95vpa,McLachlan02sm,McLachlan08evp,Xue13evp}, or on generating functions~\cite{Scovel91sni,Shang94cov,Quispel95vpi}.
For quadratic differential equations, we mention the works~\cite{Petrera11oio,Celledoni13gpo,Celledoni14ipo,Celledoni22dad,Bogfjellmo22uat} that study the Kahan-Hirota-Kimura discretization~\cite{Kahan93unm,Hirota00dote,Hirota00dotl} for the preservation of measures.
The splitting approach relies on the knowledge of the exact flows involved or depends on the dimension of the problem, and is limited to order two in the case of non-reversible problems~\cite{Blanes05otn}.
The methodology with generating functions works with any vector field, but its complexity increases with the dimension of the problem and the approach requires the evaluation of multiple integrals per step.
An important open question in geometric numerical integration is the creation of a volume-preserving method for solving a general ODE of the form~\eqref{equation:ODE} with a complexity independent of the dimension of the problem.
We investigate in this work the volume-preserving aromatic B-series method.

Introduced in~\cite{Butcher72aat,Hairer74otb} (see also the textbooks~\cite{Hairer06gni,Butcher16nmf,Butcher21bsa} and the review~\cite{McLachlan17bsa}), the Butcher-series formalism is an important tool in numerical analysis. Originally used for the calculation of order conditions for Runge-Kutta methods, the use of B-series was quickly extended to a variety of applications, in particular in geometric numerical integration~\cite{Hairer06gni}, and more recently in the approximation of stochastic evolutionary problems~\cite{Burrage96hso, Komori97rta, Rossler04ste, Debrabant08bsa, Laurent20eab} or in the theory of rough paths~\cite{Gubinelli10ror,Hairer15gvn}.
For a large class of integrators, such as Runge-Kutta methods, the modified vector field $\widetilde{f}$ in \eqref{equation:modified_ODE} is a B-series in $f$.
An extension of B-series, called aromatic B-series, was introduced in~\cite{Chartier07pfi,Iserles07bsm} for the study of volume-preserving integrators. They allow to compute the divergence of a B-series, and were later studied in~\cite{McLachlan16bsm,MuntheKaas16abs,Bogfjellmo19aso,Floystad20tup} for their algebraic and geometric properties.
In~\cite{Chartier07pfi,Iserles07bsm}, it is showed that no non-trivial B-series is divergence-free, so that the only volume-preserving B-series method is the exact flow. In particular, no Runge-Kutta method can preserve volume exactly (see also~\cite{Kang95vpa}).
However, the space of divergence-free aromatic B-series is infinite-dimensional, the simplest non-trivial example being:
\begin{equation}
\label{equation:first_ex_div_free}
\sum_{i,j,k=1}^d (f^j_k f^k_j f^i+f^j_{jk} f^k f^i-f^j_j f^i_k f^k- f^i_{jk} f^j f^k)\partial_i \in \Ker(\Div).
\end{equation}
In~\cite{MuntheKaas16abs}, the question whether there exists a volume-preserving integrator that has an expansion as an aromatic B-series is raised.
The present article gives a handful of tools to answer this question and describes explicitly the aromatic B-series of a volume-preserving integrator.
In particular, we prove that no aromatic Runge-Kutta method can preserve-volume exactly.

When using B-series, the study of volume-preservation translates into the study of the linear combinations of aromatic forests of vanishing divergence, that we call solenoidal forms. To the best of our knowledge, there does not exist any tool in the literature to describe these forests.
In differential geometry, the study of differential forms of vanishing divergence is done with the De Rham complex~\cite{Lee13its}.
For instance, in dimension 3, denote~$\Omega_n^{\text{(diff)}}(\R^3)$ the space of smooth differential~$n$-forms on~$\R^3$, then the De Rham complex writes
\begin{equation}
\label{equation:De_Rham}
\begin{tikzcd}
\Omega_0^{\text{(diff)}}(\R^3) \arrow{r}{d} & \Omega_1^{\text{(diff)}}(\R^3) \arrow{r}{d} & \Omega_2^{\text{(diff)}}(\R^3) \arrow{r}{d} & \Omega_3^{\text{(diff)}}(\R^3)
,
\end{tikzcd}
\end{equation}
where~$d$ is the exterior derivative.
In this particular example, the arrows correspond in order to the gradient, the curl, and the divergence operators.
The chain~\eqref{equation:De_Rham} is called a complex as the composition of two successive maps vanishes.
Moreover, the De Rham complex~\eqref{equation:De_Rham} is exact; that is, the image of a map is exactly the kernel of its successor. For instance, a divergence-free form~$\omega\in\Omega_3(\R^3)~$ is a curl~$\omega=d\eta$.
This exactness property is typically proven via the use of homotopy operators.
The analysis presented in this paper relies heavily on a generalisation of the De Rham complex, called the variational bicomplex, that we extend in the context of aromatic forests.

The variational bicomplex was originally introduced in the context of differential geometry~\cite{Vinogradov78ass,Tulczyjew80tel,Tsujishita82ovb, Vinogradov84tcsa,Vinogradov84tcsb} as a natural and general development of the variational chain.
It has a variety of applications in the areas of differential
geometry and topology, differential equations, mathematical physics and PDEs  (see, for instance, the textbooks~\cite{Anderson89tvb,Olver93aol}, the introductory article~\cite{Anderson92itt}, and references therein).
In this work, we introduce an algebraic tool on aromatic forests, that we call the aromatic bicomplex, in the spirit of the variational bicomplex.
If one considers the elementary differentials associated to the aromatic forests in the bicomplex, it yields a subcomplex of the standard variational bicomplex.
An originality of the approach is that the analysis  of the aromatic bicomplex uses simple combinatorics and graph theory and avoids the technical details of differential geometry.
We emphasize that, unlike the analysis of the variational bicomplex, the analysis of the aromatic bicomplex is completely independent of the dimension of the problem.
We will also add the extra assumption~$\Div(f)=0$, whose effect has never been studied, to the best of our knowledge, in the context of the variational bicomplex.
Thanks to the exactness of the aromatic bicomplex, we describe completely the solenoidal forms, both in the standard context and under the assumption~$\Div(f)=0$. We provide new operations on aromatic forests, such as the Euler operators and homotopy operators, and we draw links with the process of integration by parts of trees described in~\cite{Laurent20eab,Laurent21ocf}.
The main application of this work is the explicit description of the aromatic B-series of a volume-preserving integrator. In particular, we show that no aromatic Runge-Kutta method can preserve volume, and we propose a possible new ansatz for the creation of volume-preserving aromatic B-series methods.

The paper is organised as follows.
Section~\ref{section:preliminaries} introduces the aromatic forests and forms, the aromatic bicomplex, and presents the main theoretical results of this paper.
In Section~\ref{section:aromatic_bicomplex}, we introduce the Euler operators and study the exactness of the aromatic bicomplex, in the standard context and in the case of a divergence-free vector field.
In Section~\ref{section:applications}, we present different extensions and applications of the aromatic bicomplex.
More precisely, we introduce the augmented aromatic bicomplex and the Euler-Lagrange complex, we compute exactly the number of solenoidal forms, we derive bases and properties on divergences and solenoidal forms, and we draw links with the different integration by parts process of trees existing in the literature.
Finally, we apply our results to the study of volume-preserving integrators to obtain an explicit description of the B-series of an aromatic volume-preserving integrator.

%
%
%
%
%
%
%

\section{Preliminaries and main results}
\label{section:preliminaries}

This section is devoted to the presentation of the new objects and to the associated main algebraic results.
We first recall the definition of aromatic forests, and present an extension, called aromatic forms, well-suited for studying the divergence-free combination of forests.
Using grafting operations, we define the aromatic bicomplex, a similar tool to the variational bicomplex in the context of differential geometry, and we use its exactness to describe aromatic forms of vanishing divergence.
The proofs of the results of this section, as well as further tools, such as the Euler operators, are presented in Section~\ref{section:aromatic_bicomplex}, while we give more details and concrete applications in Section~\ref{section:applications}.

\subsection{Aromatic forms: definition and operations}

B-series were introduced by Hairer and Wanner in~\cite{Hairer74otb}, based on the work of Butcher~\cite{Butcher72aat}. Their applications in the numerical analysis of deterministic differential equations are numerous (see, for instance, the textbooks~\cite{Butcher21bsa,Hairer06gni}).
The aromatic extension of Butcher-series was introduced independently in the works~\cite{Chartier07pfi,Iserles07bsm} to study volume-preserving integrators.
In particular, this extension allows us to represent the divergence of standard B-series, which is a key tool in the study of volume-preserving methods (see \cite[Sect.\ts VI.9]{Hairer06gni}).
In the spirit of differential geometry, we work in this paper with an extension of aromatic B-series that is analogous to differential forms, and we follow the graph definition of aromatic B-series of~\cite{Bogfjellmo19aso}.

\begin{definition}
Let~$V$ be a finite set of nodes, and~$E\subset V\times V$ a set of edges. If~$a=(v,w)\in E$, the edge~$a$ is going from~$v$ to~$w$, and~$v$ is a predecessor of~$w$.
We split the set~$V$ into vertices~$V^{\bullet}$ and covertices~$V^{\circ}$. The covertices are numbered from~$1$ to~$p$, while the vertices are not numbered.
Each node in~$V$ is the target of at most one node. The nodes that are not the target of any node are called roots. The roots are numbered from~$1$ to~$n$.
Any connected component of such a graph either has a root, and is called a tree, or does not have a root, and is called an aroma.
We call aromatic forests such graphs, up to equivalence of graphs that preserve the numbering of the covertices and the roots. We write~$\FF_{n,p}$ the set of aromatic forests with~$n$ roots and~$p$ covertices,~$\FF_{n,p}^N$ its subset with forests of exactly~$N$ nodes, and~$\FF_{n}=\FF_{n,0}$. The number of nodes~$\abs{\gamma}$ is called the order of the forest~$\gamma$.
The elements of~$\FF_1$ are called aromatic trees, and the subset~$\TT$ of~$\FF_1$ that contains the trees without aromas is the set of Butcher trees.
\end{definition}

We draw the aromatic forests as follows. The vertices are represented as black nodes, and the covertices as circles of the form~$\Circled{i}$, where~$i$ is the associated number. The trees are drawn in the ascending order of their roots, from left to right. The aromas are placed in front and their order does not matter. The orientation of the edges goes from top to bottom and in clockwise order for loops. A loop with~$K$ nodes is called a~$K$-loop, in the spirit of~\cite{Iserles07bsm}.
For instance, the following forest~$\gamma\in\FF_{3,2}^{10}$ has two aromas, one~$1$-loop and one~$3$-loop,
\begin{equation}
\label{example:aromatic_forest}
\gamma=\atree3003 \atree2001 \atree3121 \atree1101 \atree1111.
\end{equation}
The aromatic forests with $N=2$ nodes are
$$\begin{array}{lll}
\FF^2_{2}=\{\atree1101 \atree1101\},
&\FF^2_{1}=\{\atree2101,\atree1001 \atree1101\},
&\FF^2_{0}=\{\atree2001,\atree2002,\atree1001\atree1001\},\\
\FF^2_{2,1}=\{\atree1101 \atree1111,\atree1111 \atree1101\},
&\FF^2_{1,1}=\{\atree2111,\atree2112,\atree1001 \atree1111,\atree1011 \atree1101\},
&\FF^2_{0,1}=\{\atree2012,\atree2013,\atree2011,\atree1001 \atree1011\},\\
\FF^2_{2,2}=\{\atree1111 \atree1121,\atree1121 \atree1111\},
&\FF^2_{1,2}=\{\atree2123,\atree2122,\atree1011 \atree1121,\atree1021 \atree1111\},
&\FF^2_{0,2}=\{\atree2021,\atree2022,\atree2023,\atree1011 \atree1021\}.
\end{array}$$

An aromatic forest in~$\FF_{n,p}$ represents a~$(n,p)$-tensor on the infinite jet bundle through the application of the elementary differential map\footnote{In the B-series literature, the elementary differential map is often defined on~$\FF_{1}$ and sends a vector field~$f$ to a vector field~$F(\gamma)(f)$. In agreement with the differential geometry literature~\cite{Anderson89tvb}, it proves more convenient to work with forms.}.
We refer the reader to~\cite{Anderson92itt} for details on the infinite jet bundle and the contact forms (see also the textbooks~\cite{Olver93aol,Kushner07cga,Lee13its}).
\begin{definition}
\label{definition:elementary_diff}
Let~$\gamma\in \FF_{n,p}$,~$f\colon\R^d\rightarrow\R^d$ a smooth vector field, and~$R=\{r_1, \dots, r_n\}\subset V$ the~$n$ roots of~$\gamma$, then the elementary differential~$F(\gamma)(f)$ is the tensor of type~$(n,p)$ on the infinite jet bundle given by
$$
F(\gamma)(f)=
\sum_{\underset{w\in V\setminus R}{i_w\in\{1,\dots,d\}}} \prod_{v\in V^{\bullet}} f^{i_v}_{I_{\Pi(v)}} dx^{i_{r_1}}\otimes \dots \otimes dx^{i_{r_n}} \otimes \theta^{i_{\Circled{1}}}_{I_{\Pi(\Circled{1})}} \otimes \dots \otimes \theta^{i_{\Circled{p}}}_{I_{\Pi(\Circled{p})}}
,
$$
where~$\Pi(v)$ is the set of predecessors of~$v\in V$ and the multi-index partial derivative is~$f^{j}_{I_S}=\frac{\partial f^{j}}{\partial x^{i_{s_1}}\dots \partial x^{i_{s_m}}}$ for~$S=\{s_1,\dots,s_m\}$.
The~$\theta^j_I$ are the contact one forms~$\theta^j_I=d f^j_I-\sum_k f^j_{I,k} dx^k$ and~$d$ is the exterior derivative.
We extend~$F$ on~$\Span(\FF_{n,p})$ by linearity.
\end{definition}
For instance, we find
$$F(\atree1001)(f)=\Div(f)=\sum_{i=1}^d f^i_i, \quad F(\atree2101)(f)=\sum_{i,j=1}^d f^i_j f^j dx^i, \quad F(\atree1111)(f)=\sum_{i=1}^d dx^i \otimes \theta^i.$$
The forest~$\gamma$ in~\eqref{example:aromatic_forest} represents the elementary differential
$$
F(\gamma)(f)=\Big(\sum_{j,k,l=1}^d f_j^k f_k^l f_l^j\Big)
\Big(\sum_{j,k=1}^d f_{jk}^j f^k\Big)
\sum_{i_{r_1},i_{r_2},i_{r_3},j,k=1}^d f^{i_{r_1}}_j f^k f^{i_{r_2}}
dx^{i_{r_1}}\otimes dx^{i_{r_2}} \otimes dx^{i_{r_3}} \otimes \theta^{i_{r_3}} \otimes \theta^j_k.
$$

In the following, we work with specific linear combinations of aromatic forests in~$\Span(\FF_{n,p})$. Given an elementary tensor, one can alternatize it to transform it into a differential form. The same idea gives rise to the aromatic forms.
\begin{definition}
For~$\gamma\in \FF_{n,p}$, let~$\SS_n^{\bullet}$ (resp.\ts~$\SS_p^{\circ}$) be the set of permutations of the roots of~$\gamma$ (resp.\ts the covertices of~$\gamma$). We define the roots wedge of~$\gamma$ as
$$
\wedge^{\bullet} \gamma = \frac{1}{n!} \sum_{\sigma\in \SS_{n}^{\bullet}} \varepsilon(\sigma) \sigma \gamma,
$$
where~$\varepsilon(\sigma)$ is the signature of the permutation~$\sigma$.
Similarly, the covertices wedge is 
$$
\wedge^{\circ} \gamma = \frac{1}{p!} \sum_{\sigma\in \SS_p^{\circ}} \varepsilon(\sigma) \sigma \gamma,
$$
and the total wedge is~$\wedge=\wedge^{\bullet}\wedge^{\circ}=\wedge^{\circ}\wedge^{\bullet}$.
\end{definition}

We extend the wedge operations to~$\Span(\FF_{n,p})$ by linearity and we denote the set of aromatic forms~$\Omega_{n,p}=\wedge\Span(\FF_{n,p})$, respectively~$\Omega_{n,p}^N=\wedge\Span(\FF_{n,p}^N)$, and~$\Omega_n=\Omega_{n,0}$.
As~$\wedge^2=\wedge$, the operator~$\wedge\colon \Span(\FF_{n,p})\rightarrow\Omega_{n,p}$ is a projection on~$\Omega_{n,p}$.

\begin{ex*}
Let~$\gamma=\atree1101 \atree2101\in \FF_2$, then
$$
\wedge \gamma=\frac{1}{2}(\atree1101 \atree2101-\atree2101 \atree1101) \in \Omega_2.$$
For~$\gamma$ given by~\eqref{example:aromatic_forest}, the associated aromatic form is
\begin{align*}
\wedge \gamma&=\frac{1}{12}\Big(
\atree3003 \atree2001 \atree3121 \atree1101 \atree1111
+\atree3003 \atree2001 \atree1111 \atree3121 \atree1101
+\atree3003 \atree2001 \atree1101 \atree1111 \atree3121
-\atree3003 \atree2001 \atree3121 \atree1111 \atree1101\\&
-\atree3003 \atree2001 \atree1101 \atree3121 \atree1111
-\atree3003 \atree2001 \atree1111 \atree1101 \atree3121
-\atree3003 \atree2001 \atree3113 \atree1101 \atree1121
-\atree3003 \atree2001 \atree1121 \atree3113 \atree1101\\&
-\atree3003 \atree2001 \atree1101 \atree1121 \atree3113
+\atree3003 \atree2001 \atree3113 \atree1121 \atree1101
+\atree3003 \atree2001 \atree1101 \atree3113 \atree1121
+\atree3003 \atree2001 \atree1121 \atree1101 \atree3113\Big).
\end{align*}
\end{ex*}


We now define the grafting and replacing operations on aromatic forests.
\begin{definition}
\label{definition:stitch_replace}
Let~$\gamma\in \FF_{n,p}$ be a forest,~$r$ a root of~$\gamma$, and~$u\in V$ (possibly equal to~$r$), then~$D^{r\rightarrow u} \gamma$ returns a copy of~$\gamma$ where the node~$r$ is now a predecessor of~$u$. We define the operator~$D^r \gamma=\sum_{u\in V} D^{r\rightarrow u} \gamma$.
Let~$\gamma\in \FF_{n,p}$ and~$v\in V^\bullet$, we define~$\gamma_{v\rightarrow \Circled{k}}$ as the forest obtained by replacing the node~$v$ by a new covertex~$\Circled{k}$. Similarly,~$\gamma_{\Circled{k}\rightarrow \bullet}$ is the forest obtained by replacing the covertex~$\Circled{k}$ by a vertex.
\end{definition}

\begin{ex*}
Let~$\gamma=\atree2111$ and~$r$ its root, then
$$
D^r \gamma=\atree2013+\atree2011,\quad
\gamma_{r\rightarrow \Circled{2}}=\atree2122,\quad
\gamma_{\Circled{1}\rightarrow \bullet}=\atree2101.
$$
\end{ex*}

The horizontal derivative is defined using grafting operations, while the vertical derivative uses the replacing operation.
\begin{definition}
Let~$\gamma\in\FF_{n,p}$, the horizontal and vertical derivatives are
$$d_H \gamma=D^{r_n}\gamma,\quad d_V \gamma=\wedge \sum_{v\in V^{\bullet}} \gamma_{v\rightarrow \Circled{p+1}}.$$
We extend~$d_H$ and~$d_V$ on~$\Omega_{n,p}$ by linearity into~$d_H\colon \Omega_{n,p} \rightarrow \Omega_{n-1,p}$ and~$d_V\colon \Omega_{n,p} \rightarrow \Omega_{n,p+1}$.
The aromatic forms in~$\Ker(d_H)$ are called solenoidal forms,~$\Psi=\Ker(d_H|_{\Omega_{1}})$ are the solenoidal combinations of trees, and~$\Psi^N=\Ker(d_H|_{\Omega_{1}^N})$ are the solenoidal combinations of trees of order~$N$.
\end{definition}
The operator~$d_H\colon \Omega_1\rightarrow\Omega_0$ is often called the divergence of an aromatic tree, as for~$\gamma\in \Omega_1$,~$d_H$ satisfies~$\Div(g)=F(d_H(\gamma))(f)$, where~$F(\gamma)(f)=\sum_i g^i dx^i$.
The operator~$d_H\colon \Omega_1\rightarrow\Omega_0$ is often called the divergence of an aromatic tree.

\begin{ex*}
Consider~$\gamma_1=\atree1101 \in \Omega_1$,~$\gamma_2=\wedge \atree1101 \atree2101\in \Omega_2$, and~$\gamma_3=\wedge \atree1101 \atree1111\in \Omega_{2,1}$, then
$$
d_H\gamma_1=\atree1001, \quad
d_H\gamma_2=\frac{1}{2}\Big(\atree2002 \atree1101+\atree2001 \atree1101-\atree1001 \atree2101-\atree3101\Big), \quad
d_H\gamma_3=\frac{1}{2}\Big(\atree1011 \atree1101+\atree2111-\atree1001 \atree1111-\atree2112\Big)
,$$
$$
d_V\gamma_1=\atree1111, \quad
d_V\gamma_2=\wedge \atree1111 \atree2101+\wedge \atree1101 \atree2111+\wedge \atree1101 \atree2112, \quad
d_V\gamma_3=\wedge\atree1121 \atree1111=\frac{1}{2}(\atree1121 \atree1111-\atree1111 \atree1121)
.$$
A calculation yields~$d_H^2\gamma_2=0$, so that~$d_H\gamma_2\in \Psi$ is a solenoidal form.
\end{ex*}

The following result, proven in Subsection~\ref{section:aromatic_bicomplex_theorem_proof}, allows us to define the bicomplex.
\begin{proposition}
\label{proposition:derivatives_squared}
The horizontal and vertical derivatives satisfy
$$
d_H^2=0, \quad d_V^2=0, \quad d_V d_H=d_H d_V.
$$
\end{proposition}

\begin{remark}
The horizontal and vertical derivatives were already defined on the set of aromatic trees~$\Omega_1=\Span(\FF_1)$ respectively in~\cite{Chartier07pfi,Iserles07bsm} for~$d_H$ and in~\cite{Floystad20tup} for~$d_V$. In this last work, the trace operator~$\Trace\colon \Omega_{1,1}\rightarrow \Omega_0$ on aromatic trees is studied. For~$\gamma \in \Omega_{1,1}$, it is given with our notations by
$$\Trace \gamma = (D^{r\rightarrow\Circled{1}}\gamma)_{\Circled{1}\rightarrow \bullet},$$
and it makes the following diagram commute.
$$\begin{tikzcd}
 \Omega_{1,1} \arrow{dr}{\Trace} & \\
 \Omega_{1} \arrow{r}{d_H} \arrow{u}{d_V} & \Omega_{0}
\end{tikzcd}$$
\end{remark}

\subsection{The aromatic bicomplex: exactness and description of solenoidal forms}
\label{section:intro_bicomplex}

The variational bicomplex is a powerful tool of variational calculus~\cite{Anderson92itt}.
We introduce in the context of aromatic forms a tool in the spirit of the variational bicomplex. This new complex, that we call the aromatic bicomplex, allows us in particular to describe explicitly the solenoidal forms in the standard context and in the divergence-free case.

The aromatic bicomplex is the diagram drawn in Figure~\ref{figure:aromatic_bicomplex}. We also introduce its variant with forms of fixed order~$N$.
\begin{figure}[!ht]
$$\begin{tikzcd}
 & \vdots & \vdots & \vdots\\
    \dots \arrow{r}{d_H} & \Omega_{2,2} \arrow{r}{d_H} \arrow{u}{d_V} & \Omega_{1,2} \arrow{r}{d_H} \arrow{u}{d_V} & \Omega_{0,2} \arrow{u}{d_V}\\
    \dots \arrow{r}{d_H} & \Omega_{2,1} \arrow{r}{d_H} \arrow{u}{d_V} & \Omega_{1,1} \arrow{r}{d_H} \arrow{u}{d_V} & \Omega_{0,1} \arrow{u}{d_V}\\
    \dots \arrow{r}{d_H} & \Omega_{2} \arrow{r}{d_H} \arrow{u}{d_V} & \Omega_{1} \arrow{r}{d_H} \arrow{u}{d_V} & \Omega_{0} \arrow{u}{d_V}\\
     & 0 \arrow{u} & 0 \arrow{u} & 0 \arrow{u}
\end{tikzcd}
\qquad\qquad
\begin{tikzcd}
 & 0 &  & 0\\
0 \arrow{r} & \Omega_{N,N}^N \arrow{r}{d_H} \arrow{u} & \dots \arrow{r}{d_H} & \Omega_{0,N}^N \arrow{u}\\
& \vdots \arrow{u}{d_V} &  & \vdots \arrow{u}{d_V}\\
0 \arrow{r} & \Omega_{N}^N \arrow{r}{d_H} \arrow{u}{d_V} & \dots \arrow{r}{d_H} & \Omega_{0}^N \arrow{u}{d_V}\\
 & 0 \arrow{u} &  & 0 \arrow{u}
\end{tikzcd}$$
\caption{The aromatic bicomplex (left) and its subcomplex of order~$N$ (right).}
	\label{figure:aromatic_bicomplex}
\end{figure}
The aromatic bicomplex can be completed by an extra column on the right in order, for instance, to describe the aromatic forms in~$\Img(d_H|_{\Omega_{1,p}})$, as we will see in Subsection~\ref{section:augmented_bicomplex}.
We refer the reader to the appendix~\ref{section:examples_aromatic_bicomplex} for examples of the (augmented) aromatic bicomplex for the first values of~$N$.

\begin{remark}
The elementary differential map sends the aromatic bicomplex to a subcomplex of the variational bicomplex~\cite{Anderson92itt} in the following way.
Consider a vector field~$f\colon \R^d\rightarrow \R^d$, the space~$\Omega_{n,p}^{\text{(diff)}}$ of differential forms of type~$(n,p)$, and the Hodge star operator~$\star\colon \Omega_{n,p}^{\text{(diff)}}\rightarrow\Omega_{d-n,p}^{\text{(diff)}}$ (see, for instance,~\cite{Lee13its}).
The map~$\star F(.)(f)$ sends an aromatic form in~$\Omega_{n,p}$ to a differential form in~$\Omega_{d-n,p}^{\text{(diff)}}$. The derivatives~$d_H$,~$d_V$ on the aromatic bicomplex and~$d_H^{\text{(diff)}}$,~$d_V^{\text{(diff)}}$ on the variational bicomplex satisfy
$$d_H^{\text{(diff)}} \star F(\gamma)(f)= \star F(d_H\gamma)(f), \quad
d_V^{\text{(diff)}} F(\gamma)(f)=(-1)^{n+p} F(d_V\gamma)(f).$$
Note that the dimension~$d$ of the problem plays a role in the context of differential geometry, but not for aromatic forms.
Note also that~$d_H$ and~$d_V$ commute (see Proposition~\ref{proposition:derivatives_squared}), while their counterparts from differential geometry anticommute.
%
\end{remark}

The main property of the aromatic bicomplex is its exactness.
\begin{theorem}
\label{theorem:exact_aro_var_bi}
The horizontal and vertical sequences of the aromatic bicomplex are exact, that is, for all~$n$,~$p\geq 0$,
$$
\Img(d_H|_{\Omega_{n+1,p}^N})=\Ker(d_H|_{\Omega_{n,p}^N}), \quad \Img(d_V|_{\Omega_{n,p}^N})=\Ker(d_V|_{\Omega_{n,p+1}^N}).
$$
\end{theorem}

With Theorem~\ref{theorem:exact_aro_var_bi}, it is straightforward to generate all the solenoidal aromatic forms by considering the~$d_H \wedge \gamma$ for~$\gamma\in \FF_2$.
For example, the only basis element of~$\Psi^3$ (which corresponds to the vector field \eqref{equation:first_ex_div_free}) is
\begin{equation}
\label{equation:solenoidal_order_three}
2d_H\wedge \atree1101 \atree2101=\atree2002 \atree1101+\atree2001 \atree1101-\atree1001 \atree2101-\atree3101,
\end{equation}
and a basis of~$\Psi^4$ is given by the forests
\begin{align*}
2d_H\wedge \atree1101 \atree3102 &=\atree3001 \atree1101+\atree3002 \atree1101+\atree3003 \atree1101-\atree4103-\atree4104-\atree1001 \atree3102,\\
2d_H\wedge \atree1101 \atree3101 &=\atree3004 \atree1101+2\atree3002 \atree1101+\atree4103-2\atree4104-\atree4101-\atree1001 \atree3101,\\
2d_H\wedge \atree1001 \atree1101 \atree2101 &=\atree3001 \atree1101+\atree1001 \atree2002 \atree1101+\atree1001 \atree2001 \atree1101-\atree2001 \atree2101-\atree1001 \atree1001 \atree2101-\atree1001 \atree3101.
\end{align*}
However, the set~$\{d_H \wedge \gamma, \gamma\in \FF_2\}$ does not form a basis of the solenoidal forms in general. We give a basis of~$\Psi$ in Subsection~\ref{section:description_image_kernel_dH} alongside bases of the image and kernel of~$d_H$ and its dual~$d_H^*$.

For the study of volume-preserving integrators for solving~\eqref{equation:ODE}, it is fundamental to assume that the vector field~$f$ satisfies~$\Div(f)=0$.
With aromatic forms, it amounts to sending all forests containing an aroma with a 1-loop to~$0$.
We consider the vector space~$\AA$ spanned by aromatic forests containing at least one 1-loop,~$\AA_{n,p}=\AA\cap\Omega_{n,p}$, and~$\AA_{n,p}^N=\AA\cap\Omega_{n,p}^N$.
We write~$\widetilde{\FF}_{n,p}$ the set of aromatic forests in~$\FF_{n,p}$ without 1-loops,~$\widetilde{\Omega}_{n,p}=\Omega_{n,p}/\AA_{n,p}$,~$\widetilde{\Omega}_{n,p}^N=\Omega_{n,p}^N/\AA_{n,p}^N$,~$\widetilde{\Omega}_n=\widetilde{\Omega}_{n,0}$,~$\widetilde{\Psi}=\Ker(d_H|_{\widetilde{\Omega}_{1}})$, and~$\widetilde{\Psi}^N=\Ker(d_H|_{\widetilde{\Omega}_{1}^N})$.
The divergence-free aromatic bicomplex with~$N$ nodes is drawn in Figure~\ref{figure:div_free_bicomplex}.
\begin{figure}[!ht]
$$\begin{tikzcd}
 & \vdots & \vdots & \vdots\\
    \dots \arrow{r}{d_H} & \widetilde{\Omega}_{2,2}^N \arrow{r}{d_H} \arrow{u}{d_V} & \widetilde{\Omega}_{1,2}^N \arrow{r}{d_H} \arrow{u}{d_V} & \widetilde{\Omega}_{0,2}^N \arrow{u}{d_V}\\
    \dots \arrow{r}{d_H} & \widetilde{\Omega}_{2,1}^N \arrow{r}{d_H} \arrow{u}{d_V} & \widetilde{\Omega}_{1,1}^N \arrow{r}{d_H} \arrow{u}{d_V} & \widetilde{\Omega}_{0,1}^N \arrow{u}{d_V}\\
    \dots \arrow{r}{d_H} & \widetilde{\Omega}_{2}^N \arrow{r}{d_H} \arrow{u}{d_V} & \widetilde{\Omega}_{1}^N \arrow{r}{d_H} \arrow{u}{d_V} & \widetilde{\Omega}_{0}^N \arrow{u}{d_V}\\
    & 0 \arrow{u} & 0 \arrow{u} & 0 \arrow{u}
\end{tikzcd}
\qquad\qquad
\begin{tikzcd}
 & 0 & 0 \\
0 \arrow{r} & \widetilde{\Omega}_{1,1}^1=\Span(\atree1111) \arrow{r}{d_H} \arrow{u} & \widetilde{\Omega}_{0,1}^1=0 \arrow{u}\\
0 \arrow{r} & \widetilde{\Omega}_{1}^1=\Span(\atree1101) \arrow{r}{d_H} \arrow{u}{d_V} & \widetilde{\Omega}_{0}^1=0 \arrow{u}{d_V}\\
 & 0 \arrow{u} & 0 \arrow{u}
\end{tikzcd}$$
\caption{The divergence-free aromatic bicomplex of order~$N$ (left) and of order~$N=1$ (right).}
	\label{figure:div_free_bicomplex}
\end{figure}
In the simplest case~$N=1$ (see Figure~\ref{figure:div_free_bicomplex}), the aromatic bicomplex is not exact.
Indeed, we have~$\atree1101\in \Ker(d_H)$, but~$\atree1101\notin \Img(d_H)$.
One of the main results of this paper is that the case~$N=1$ is the only case where the divergence-free aromatic bicomplex is not exact.
\begin{theorem}
\label{theorem:exact_div_free_complex}
The divergence-free aromatic bicomplex with~$N$ nodes is exact if and only if~$N\neq 1$.
\end{theorem}

The main consequence of the exactness of the aromatic bicomplex in both contexts is that the assumption~$\Div(f)=0$ does not create new non-trivial solenoidal forms.
\begin{theorem}
\label{theorem:Kernel_divergence_free}
For~$N\neq 1$, and all~$n\geq 1$,~$p\geq 0$, the kernel of the divergence operator~$d_H$ satisfies
$$\Ker(d_H|_{\widetilde{\Omega}_{n,p}^N})=\Ker(d_H|_{\Omega_{n,p}^N})/\AA_{n,p}^N,$$
that is, the solenoidal aromatic forms in the divergence-free context exactly correspond to the solenoidal aromatic forms in the standard context, except for the forms spanned by~$\atree1101$ and~$\atree1111$.
\end{theorem}

\begin{proof}[Proof of Theorem~\ref{theorem:Kernel_divergence_free}]
As~$d_H$ does not decrease the number of~$1$-loops in a forest, the image of~$d_H$ satisfies
$$\Img(d_H|_{\widetilde{\Omega}_{n+1,p}^N})=\Img(d_H|_{\Omega_{n+1,p}^N})/\AA_{n,p}^N.$$
Theorem~\ref{theorem:exact_aro_var_bi} and Theorem~\ref{theorem:exact_div_free_complex} yield the desired identity.
\end{proof}

A generating set of the solenoidal forms~$\widetilde{\Psi}^N$ of order~$N>1$ is obtained by deleting the~$1$-loops in the generating set~$\{d_H \wedge \gamma, \gamma\in \FF_2^N\}$ of~$\Psi^N$.
For instance, for order~$N=3$, we derive from the element~\eqref{equation:solenoidal_order_three} that the solenoidal forms in the divergence-free context are given by
$$\widetilde{\Psi}^3=\Span(\atree2002 \atree1101-\atree3101).$$
We give generators of the solenoidal forms~$\widetilde{\Psi}^N$ for the first orders in Appendix~\ref{section:solenoidal_forms}.
This explicit description of solenoidal forms is especially useful in the numerical study of volume-preserving integrators, as discussed in Subsection~\ref{section:application_vp_integrators}.

We enumerate in Subsection~\ref{section:combinatorics_Robert} the dimensions of the~$\Omega_{n,p}^N$ and~$\widetilde{\Omega}_{n,p}^N$ in the first two rows of the bicomplex and deduce the dimension of~$\Psi^N$ in Theorem~\ref{theorem:counting_standard} and of~$\widetilde{\Psi}^N$ in Theorem~\ref{theorem:counting_div_free}.
A surprising fundamental result is that the solenoidal forms in~$\Psi^N$ are enumerated by the difference between the number of aromatic trees in~$\Omega_{1}^N$ and the number of aromas with 1-loops~$\mathring{\Omega}_{0}^N$ (see Table~\ref{table:dimensions_solenoidal} for examples).

\begin{table}[!htb]
	\setcellgapes{3pt}
	\centering
	\begin{tabular}{|c|c|c|c|c|c|c|c|c|c|c|c|c|c|c|}
	\hline
	$N$ &$1$ &$2$ &$3$ &$4$ &$5$ &$6$ &$7$ &$8$ &$9$ &$10$ &$11$ &$12$ &$13$ &$14$\\
	\hline
	$|\Omega_{1}^N|$ &$1$ &$2$ &$6$ &$16$ &$45$ &$121$ &$338$ &$929$ &$2598$ &$7261$ &$20453$ &$57738$ &$163799$ &$465778$\\
	\hline
	$|\mathring{\Omega}_{0}^N|$ &$1$ &$2$ &$5$ &$13$ &$34$ &$90$ &$243$ &$660$ &$1818$ &$5045$ &$14102$ &$39639$ &$111982$ &$317533$\\
	\hline
	$|\Psi^N|$ &$0$ &$0$ &$1$ &$3$ &$11$ &$31$ &$95$ &$269$ &$780$ &$2216$ &$6351$ &$18099$ &$51817$ &$148245$\\
	\hline
	$|\widetilde{\Psi}^N|$ &$1$ &$0$ &$1$ &$2$ &$7$ &$16$ &$48$ &$123$ &$346$ &$937$ &$2626$ &$7284$ &$20533$ &$57804$\\
	\hline
	\end{tabular}
	\caption{Dimensions of the space of solenoidal forms in both contexts for the first orders~$N$ (see Theorems~\ref{theorem:counting_standard} and~\ref{theorem:counting_div_free}).
	Note how~$|\Psi^N|=|\Omega_{1}^N|-|\mathring{\Omega}_{0}^N|$.}
	\label{table:dimensions_solenoidal}
	\setcellgapes{1pt}
\end{table}

\section{The aromatic bicomplex}
\label{section:aromatic_bicomplex}

This section is devoted to the study of the aromatic bicomplex.
First, we introduce the Euler operators and prove the exactness of the variational chain in Subsection~\ref{section:Euler_operators}.
We define the horizontal and vertical homotopy operators and prove the exactness of the aromatic bicomplex in Subsection~\ref{section:aromatic_bicomplex_theorem_proof}.
We study the aromatic bicomplex in the divergence-free context in Subsection~\ref{section:divergence-free_bicomplex}.

\subsection{The Euler operators and the variational complex}
\label{section:Euler_operators}

To define the Euler operators, we extend Definition~\ref{definition:stitch_replace} to allow one to detach edges and graft them back.
\begin{definition}
\label{definition:general_stitch}
A graph~$\gamma$ is a marked aromatic forest if it is an aromatic forest with exactly one of its node that holds the symbol~$\diamond$. We write~$\FF_{n,p}^\diamond$ the set of marked aromatic forests with~$n$ roots and~$p$ covertices.
Given~$\gamma\in \FF_{n,p}^\diamond$,~$\gamma^\diamond\in \FF_{n,p}$ is the same forest~$\gamma$ where the symbol~$\diamond$ is removed.

Let~$\gamma\in \FF_{n,p}$ or~$\gamma\in \FF_{n,p}^\diamond$, where~$v_\diamond\in V$ is the node with the symbol~$\diamond$. Let~$R_0\subset R$ be a given subset of roots that we call the set of detached nodes.
For~$\gamma\in \FF_{n,p}^\diamond$,~$r\in R$ and~$u\in V$ with~$u\neq v_\diamond$,~$D^{r\rightarrow u} \gamma\in \FF_{n-1,p}^\diamond$ is the marked aromatic forest obtained by adding the edge linking~$r$ to~$u$ to~$\gamma$. The set of detached roots of~$D^{r\rightarrow u} \gamma$ is~$R_0$ if~$r\notin R_0$ and~$R_0\setminus \{r\}$ else.
We define~$$D^r \gamma=\sum_{u\in V\setminus \{v_\diamond\}} D^{r\rightarrow u} \gamma.$$
Let~$q$ be a non-negative integer,~$u\in V$ with~$u\neq v_\diamond$ if~$\gamma\in \FF_{n,p}^\diamond$, and~$I\subset R_0$. We define~$D^I$ and~$D^{I\rightarrow u}$ by 
$$D^I \gamma=\sum_{\varphi\colon I\rightarrow V} \prod_{w\in I} D^{w\rightarrow \varphi(w)} \gamma,
\quad
D^{I\rightarrow u} \gamma=\prod_{w\in I} D^{w\rightarrow u} \gamma,
$$
and~$D^q$ and~$D^{q\rightarrow u}$ by
$$D^q \gamma=\sum_{I\subset R_0, \abs{I}=q} D^I \gamma,
\quad
D^{q\rightarrow u} \gamma=\sum_{I\subset R_0, \abs{I}=q} D^{I\rightarrow u}\gamma.
$$

Let~$\gamma\in \FF_{n,p}$ and~$v\in V$, then~$\gamma_{v^{\starsymbol}}\in \FF_{n+\abs{\Pi(v)},p}$ is the graph obtained by cutting off the edges of~$\gamma$ pointing to~$v$ and placing the newly obtained roots after the roots in~$R$ (in an arbitrary order). In addition, we add a symbol~$\starsymbol$ on the node~$v$ and we fix~$R_0=\Pi(v)$ as the set of detached nodes.
\end{definition}

The tools from Definition~\ref{definition:general_stitch} are extended by linearity. In simple words, the~$\diamond$ detaches the predecessors of a given node~$v_\diamond$ and put them in a set~$R_0$. As long as the symbol~$\diamond$ is present, the grafting operators~$D^I$ and~$D^q$ graft the detached nodes on every node except~$v_\diamond$.
In particular, the following operation is trivial:
$$D^{\abs{\Pi(v)}\rightarrow v}(\gamma_{v^{\starsymbol}})^{\starsymbol}=\gamma.$$

\begin{ex*}
Consider the forest~$\gamma=\,\includegraphics[scale=0.48]{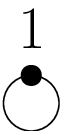}\, \,\includegraphics[scale=0.48]{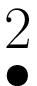}\,$ that we label for the sake of clarity.
The marked forest~$\gamma_{1^{\starsymbol}}$ has the set of detached nodes~$R_0=\{1\}$ and satisfies
$$
\gamma_{1^{\starsymbol}}=\,\includegraphics[scale=0.48]{Aromatic_forms/Labtree_1101.eps}\,\,\includegraphics[scale=0.48]{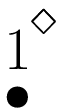}\,  ,\quad
(D\gamma_{1^{\starsymbol}})^{\diamond}=\,\includegraphics[scale=0.48]{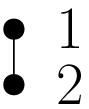}\,  ,\quad
D(\gamma_{1^{\starsymbol}})^{\diamond}=\,\includegraphics[scale=0.48]{Aromatic_forms/Labtree_2101.eps}\,+\,\includegraphics[scale=0.48]{Aromatic_forms/Labtree_1001.eps}\, \,\includegraphics[scale=0.48]{Aromatic_forms/Labtree_1101.eps}\,  .
$$
\end{ex*}

For simplicity, we write~$\gamma_{v\rightarrow \Circled{k}^{\starsymbol}}$ for~$(\gamma_{v\rightarrow \Circled{k}})_{\Circled{k}^{\starsymbol}}$, and~$D=D^1$ in the rest of the paper.
Note that, in general,~$D^2$ and~$D\circ D$ are different operators. Instead, straightforward combinatorics yield
\begin{equation}
\label{equation:identity_composition_derivatives}
\binom{n}{p} D^n= D^p D^{n-p}.
\end{equation}
However, for~$u$,~$v\in R$, the operations~$D^u$ and~$D^v$ commute.

\begin{remark}
\label{remark:clarification_notation}
The notations introduced in Definition~\ref{definition:general_stitch} closely depend on the forest~$\gamma$ they are applied to. For the sake of clarity, in the rest of the paper, the notations shall always relate to the original forest denoted~$\gamma$.
For instance, in~$(-1)^{\abs{\Pi(v)}} \gamma_{v^{\starsymbol}}$,~$\Pi(v)$ denotes the number of predecessors of~$v$ in~$\gamma$, not in~$\gamma_{v^{\starsymbol}}$.
\end{remark}

In the context of differential geometry, the Euler operator describes the differential forms that are divergences~\cite{Olver93aol}. It has a variety of applications such as the computation of conservation laws for PDEs~\cite{Hereman05cad,Hereman07cad,Poole10tho}.
We define a similar operator for aromatic forms.
\begin{definition}
For~$\gamma\in \FF_{n,p}$ and~$v\in V$, the Euler operators~$\Eulerop_v \colon \FF_{n,p}\rightarrow \Omega_{n,p}$ and~$\Eulerop_v^{\circ}\colon \FF_{0}\rightarrow \Omega_{0,1}$ are given by
$$
\Eulerop_v \gamma=(-1)^{\abs{\Pi(v)}} (D^{\Pi(v)} \gamma_{v^{\starsymbol}})^{\starsymbol},
\quad
\Eulerop_v^{\circ} \gamma=(-1)^{\abs{\Pi(v)}} (D^{\Pi(v)} \gamma_{v\rightarrow \Circled{1}^{\starsymbol}})^{\starsymbol}
$$
and are extended by linearity on~$\Omega_{n,p}$.
The Euler operator~$\Eulerop \colon \Omega_{n,p}\rightarrow \Omega_{n,p}$ and its variant~$\Eulerop^{\circ}\colon \Omega_{0}\rightarrow \Omega_{0,1}$ are given by
$$
\Eulerop\gamma=\sum_{v\in V} \Eulerop_v \gamma, \quad
\Eulerop^{\circ}\gamma=\sum_{v\in V} \Eulerop_v^{\circ} \gamma.
$$
\end{definition}

The output of~$\Eulerop^{\circ}$ on the forms of first orders is
$$\Eulerop^{\circ} \atree1001=0,\quad
\Eulerop^{\circ} \atree2001=2\atree2013,\quad
\Eulerop^{\circ} \atree2002=-2\atree2013,\quad
\Eulerop^{\circ} \atree1001\atree1001=-2\atree2013.$$
From these examples, we observe that the Euler operator vanishes on aromatic forms that are divergences:
$$\Eulerop^{\circ} d_H \atree1101=0,\quad
\Eulerop^{\circ} d_H \atree2101=0,\quad
\Eulerop^{\circ} d_H \atree1001\atree1101=0.$$

\begin{proposition}
\label{proposition:L_d_vanish}
Let~$\gamma\in \FF_{n,p}$ and~$v\in V$, the Euler operators satisfy
$$\Eulerop_v d_H\gamma=0,\quad \Eulerop d_H\gamma=0$$
and for~$p=0$,~$\Eulerop^{\circ}$ satisfies~$\Eulerop^{\circ} d_H\gamma=0.$
\end{proposition}

\begin{proof}
Let~$\gamma\in \FF_{n,p}$ and~$r_n$ be its last root, then
\begin{align*}
\Eulerop_v  d_H\gamma&=\sum_{u\in V} \Eulerop_v D^{r_n\rightarrow u} \gamma\\
&=(-1)^{\abs{\Pi(v)}+1} (D^{\Pi(v)} D^{r_n} \gamma_{v^{\starsymbol}})^{\starsymbol}
+\sum_{u\neq v} (-1)^{\abs{\Pi(v)}} (D^{\Pi(v)} D^{r_n\rightarrow u} \gamma_{v^{\starsymbol}})^{\starsymbol}
=0,
\end{align*}
where the two terms correspond to the cases~$u=v$ and~$u\neq v$, and where we used the convention of Remark~\ref{remark:clarification_notation}. The identities with~$\Eulerop$ and~$\Eulerop^{\circ}$ are direct consequences.
\end{proof}

We know that the composition of the two maps~$\Eulerop^{\circ}$ and~$d_H$ vanishes. One is then interested in a necessary and sufficient condition for an aromatic form~$\gamma\in \Omega_0$ to be a divergence. The following chain is called the variational complex.
\begin{equation}
\label{equation:variational_complex}
\begin{tikzcd}
	\Omega_{1} \arrow{r}{d_H} & \Omega_{0}\arrow{r}{\Eulerop^{\circ}} & \Omega_{0,1}
\end{tikzcd}
\end{equation}
The complex~\eqref{equation:variational_complex} is exact if~$\Img(d_H)=\Ker(\Eulerop^{\circ})$; that is, if~$\gamma\in \Omega_0$ is a divergence if and only if~$\Eulerop^{\circ} \gamma=0$.
A fundamental question in variational calculus~\cite{Olver93aol} is the exactness of this chain (in the context of differential forms).
We prove in the rest of this subsection the exactness of the variational complex~\eqref{equation:variational_complex} in the context of aromatic forms.
We rely on the use of the Euler operators of higher orders and homotopy operators.
To the best of our knowledge, the use of such operators on aromatic forests is completely new.
\begin{definition}
For~$\gamma\in\FF_{n,p}$, the Euler operator~$\Eulerop^q_v \gamma$ of order~$q\geq 0$ on~$v\in V$ is
$$
\Eulerop^q_v \gamma=(-1)^{\abs{\Pi(v)}-q} (D^{\abs{\Pi(v)}-q} \gamma_{v^{\starsymbol}})^{\starsymbol},
$$
that we extend on~$\Omega_{n,p}$ by linearity.
The higher Euler operators are
$$
\Eulerop^q \gamma=\sum_{v\in V} \Eulerop^q_v \gamma.
$$
\end{definition}

A fundamental result for our analysis is that any aromatic forest can be rewritten with the Euler operators. Note that all appearing series are finite, as~$\Eulerop^q_v \gamma=0$ if~$q> \abs{\gamma}$. Note also that~$\Eulerop^0_v=\Eulerop_v$.
\begin{proposition}
\label{prop:higher_Euler_operators_Omega}
The Euler operators satisfy for all~$\gamma\in\FF_{n,p}$ and all vertices~$v\in V$,
\begin{equation}
\label{equation:decomposition_Euler_op_v}
\gamma=\sum_{q=0}^\infty D^{q} \Eulerop^q_v \gamma.
\end{equation}
In particular, we have
\begin{equation}
\label{equation:decomposition_Euler_op}
\abs{\gamma}\gamma=\sum_{q=0}^\infty D^q \Eulerop^q \gamma.
\end{equation}
In addition, if~$\gamma$ satisfies~$\Eulerop^{p}\gamma=0$ for~$p=0,\dots,n-1$, then~$\gamma=D^n(h^{(n)}\gamma)$, where
$$
h^{(n)}\gamma=\frac{1}{\abs{\gamma}}\sum_{p=n}^\infty \binom{p}{n}^{-1} D^{p-n} \Eulerop^{p} \gamma.
$$
\end{proposition}

The proof shares similarities with the approach of~\cite{Hereman07cad} and uses the following intermediate result in the spirit of the Leibniz rule.
\begin{lemma}
\label{lemma:Leibniz_divergence}
Let~$\gamma\in \FF_{n,p}$ and~$v\in V$, let three integers~$p$,~$q$,~$k$ such that~$p+q+k=\abs{\Pi(v)}-1$. Then, the following holds:
\begin{align*}
(q+1)D^{q+1}D^{k\rightarrow v}(D^p\gamma_{v^{\starsymbol}})^{\starsymbol}
&=(k+1)D^{q}D^{k+1\rightarrow v}(D^p\gamma_{v^{\starsymbol}})^{\starsymbol}
+(p+1)D^{q}D^{k\rightarrow v}(D^{p+1}\gamma_{v^{\starsymbol}})^{\starsymbol}. \nonumber
\end{align*}
More precisely, we have
\begin{equation}
\label{equation:composition_derivatives}
D^q D^{k\rightarrow v} (D^{p} \gamma_{v^{\starsymbol}})^{\starsymbol}
=\sum_{n=0}^k (-1)^{k-n} \binom{q+n}{q}\binom{p+k-n}{p} D^{q+n} (D^{p+k-n} \gamma_{v^{\starsymbol}})^{\starsymbol}.
\end{equation}
\end{lemma}

\begin{proof}[Proof of Proposition~\ref{prop:higher_Euler_operators_Omega}]
Let~$\gamma\in\FF_{n,p}$. With the help of Lemma~\ref{lemma:Leibniz_divergence}, we get
\begin{align*}
\Eulerop_v\gamma
&=(-1)^{\abs{\Pi(v)}} (D^{\abs{\Pi(v)}} \gamma_{v^{\starsymbol}})^{\starsymbol}\\
&= (\gamma_{v^{\starsymbol}})^{\starsymbol} \ind_{\abs{\Pi(v)}=0}
- (-1)^{\abs{\Pi(v)}-1} (D^{\abs{\Pi(v)}} \gamma_{v^{\starsymbol}})^{\starsymbol} \ind_{\abs{\Pi(v)}\geq 1}\\
&=\gamma \ind_{\abs{\Pi(v)}=0}
-\frac{(-1)^{\abs{\Pi(v)}-1}}{\abs{\Pi(v)}} D (D^{\abs{\Pi(v)}-1} \gamma_{v^{\starsymbol}})^{\starsymbol}
+\frac{(-1)^{\abs{\Pi(v)}-1}}{\abs{\Pi(v)}} D^{1\rightarrow v}(D^{\abs{\Pi(v)}-1} \gamma_{v^{\starsymbol}})^{\starsymbol}
\end{align*}
We apply the same reasoning to the last term
\begin{align*}
\frac{(-1)^{\abs{\Pi(v)}-1}}{\abs{\Pi(v)}} D^{1\rightarrow v}(D^{\abs{\Pi(v)}-1} \gamma_{v^{\starsymbol}})^{\starsymbol}
&=\gamma \ind_{\abs{\Pi(v)}= 1}
-\frac{(-1)^{\abs{\Pi(v)}-2}}{\abs{\Pi(v)}(\abs{\Pi(v)}-1)} D D^{1\rightarrow v}(D^{\abs{\Pi(v)}-2} \gamma_{v^{\starsymbol}})^{\starsymbol}\\
&+\frac{2(-1)^{\abs{\Pi(v)}-2}}{\abs{\Pi(v)}(\abs{\Pi(v)}-1)} D^{2\rightarrow v} (D^{\abs{\Pi(v)}-2} \gamma_{v^{\starsymbol}})^{\starsymbol}.
\end{align*}
Iterating this reasoning, we find
\begin{equation}
\label{equation:intermediate_expression_decomposition_Euler_op}
\Eulerop_v\gamma=\gamma
-\sum_{p=1}^\infty \frac{(-1)^{\abs{\Pi(v)}-p}}{p\binom{\abs{\Pi(v)}}{p}} D D^{p-1\rightarrow v} (D^{\abs{\Pi(v)}-p} \gamma_{v^{\starsymbol}})^{\starsymbol}.
\end{equation}
Using the formula~\eqref{equation:composition_derivatives} in~\eqref{equation:intermediate_expression_decomposition_Euler_op}, standard combinatorics~\cite{Merris03com} yield
\begin{align*}
\gamma&=\Eulerop_v\gamma + \sum_{p=1}^\infty \frac{(-1)^{\abs{\Pi(v)}-p}}{p\binom{\abs{\Pi(v)}}{p}} 
\sum_{n=0}^{p-1} (-1)^{p-n-1} (n+1) \binom{\abs{\Pi(v)}-n-1}{\abs{\Pi(v)}-p} D^{n+1} (D^{\abs{\Pi(v)}-n-1} \gamma_{v^{\starsymbol}})^{\starsymbol}\\
&=\Eulerop_v\gamma + 
\sum_{n=0}^\infty
\sum_{p=n+1}^\infty \frac{n+1}{p\binom{\abs{\Pi(v)}}{p}}  \binom{\abs{\Pi(v)}-n-1}{\abs{\Pi(v)}-p} D^{n+1} \Eulerop_v^{n+1}\gamma\\
&=\Eulerop_v\gamma + 
\sum_{n=1}^{\abs{\Pi(v)}}
\sum_{p=n}^{\abs{\Pi(v)}} \frac{n(\abs{\Pi(v)}-n)!(p-1)!}{\abs{\Pi(v)}!(p-n)!} D^{n} \Eulerop_v^{n}\gamma\\
&=\Eulerop_v\gamma + 
\sum_{n=1}^{\abs{\Pi(v)}} \binom{\abs{\Pi(v)}}{n}^{-1}
\sum_{p=0}^{\abs{\Pi(v)}-n} \binom{p+n-1}{p} D^{n} \Eulerop_v^{n}\gamma
=\sum_{n=0}^\infty D^{n} \Eulerop^n_v \gamma.
\end{align*}
We sum~\eqref{equation:decomposition_Euler_op_v} on all the nodes~$v\in V$ to obtain~\eqref{equation:decomposition_Euler_op}.
The last claim of Proposition~\ref{prop:higher_Euler_operators_Omega} is obtained by using the formula~\eqref{equation:identity_composition_derivatives} in~\eqref{equation:decomposition_Euler_op}.
\end{proof}

\begin{proof}[Proof of Lemma~\ref{lemma:Leibniz_divergence}]
Using the identity~\eqref{equation:identity_composition_derivatives}, we distribute the derivatives on~$v$ and the other nodes.
\begin{align*}
D^{q+1}D^{k\rightarrow v}(D^p\gamma_{v^{\starsymbol}})^{\starsymbol}
&= \frac{1}{q+1} D D^q D^{k\rightarrow v}(D^p\gamma_{v^{\starsymbol}})^{\starsymbol}\\
&=\frac{1}{q+1} \sum_{\underset{\abs{S_1}=q,\abs{S_2}=k,\abs{S_3}=p}{S_1\sqcup S_2\sqcup S_3\sqcup \{u\}=\Pi(v)}}  D^{u} D^{S_1} D^{S_2\rightarrow v} (D^{S_3}\gamma_{v^{\starsymbol}})^{\starsymbol}\\
&=\frac{1}{q+1} \sum_{\underset{\abs{S_1}=q,\abs{S_2}=k,\abs{S_3}=p}{S_1\sqcup S_2\sqcup S_3\sqcup \{u\}=\Pi(v)}} D^{S_1} \bigg[ D^{S_2\cup u \rightarrow v} (D^{S_3}\gamma_{v^{\starsymbol}})^{\starsymbol}
+  D^{S_2\rightarrow v} (D^{S_3\cup u}\gamma_{v^{\starsymbol}})^{\starsymbol}\bigg]\\
&=\frac{k+1}{q+1} \sum_{\underset{\abs{S_1}=q,\abs{S_2}=k+1,\abs{S_3}=p}{S_1\sqcup S_2\sqcup S_3=\Pi(v)}}  D^{S_1} D^{S_2 \rightarrow v} (D^{S_3}\gamma_{v^{\starsymbol}})^{\starsymbol}\\
&+\frac{p+1}{q+1} \sum_{\underset{\abs{S_1}=q,\abs{S_2}=k,\abs{S_3}=p+1}{S_1\sqcup S_2\sqcup S_3=\Pi(v)}}  D^{S_1} D^{S_2\rightarrow v} (D^{S_3}\gamma_{v^{\starsymbol}})^{\starsymbol}\\
&=\frac{k+1}{q+1} D^{q}D^{k+1\rightarrow v}(D^p\gamma_{v^{\starsymbol}})^{\starsymbol} +\frac{p+1}{q+1} D^{q}D^{k\rightarrow v}(D^{p+1}\gamma_{v^{\starsymbol}})^{\starsymbol}.
\end{align*}
We obtain the formula~\eqref{equation:composition_derivatives} by induction on~$k$.
\end{proof}

A direct consequence of Proposition~\ref{prop:higher_Euler_operators_Omega} is the exactness of the variational complex~\eqref{equation:variational_complex}.
\begin{theorem}
\label{theorem:exact_variational_chain}
For~$\gamma\in \FF_{0,1}$, let the variational homotopy operator~$h_V$ be
$$
h_V \gamma=\frac{1}{\abs{\gamma}} \gamma_{\Circled{1}\rightarrow \bullet},
$$
and for~$\gamma\in \FF_0$, let the horizontal homotopy operator~$h_H$ be
$$
h_H \gamma=\frac{1}{\abs{\gamma}} \sum_{q=1}^\infty \frac{1}{q} D^{q-1} \Eulerop^{q} \gamma.
$$
We extend the definition on~$\Omega_{0,1}$ and~$\Omega_0$ by linearity.
These operators satisfy for all~$\gamma\in\Omega_0$,
\begin{equation}
\label{equation:homotopy_equality_variational}
(d_H h_H +h_V \Eulerop^{\circ})\gamma=\gamma.
\end{equation}
In particular, the variational complex~\eqref{equation:variational_complex} is exact.
\end{theorem}

\begin{proof}
Let~$\gamma\in \Omega_0$. From Proposition~\ref{prop:higher_Euler_operators_Omega}, we obtain 
$$
\gamma=\frac{1}{\abs{\gamma}} \sum_{q=0}^\infty D^q \Eulerop^q \gamma
=\frac{1}{\abs{\gamma}} \Eulerop \gamma
+\frac{1}{\abs{\gamma}} \sum_{q=1}^\infty \frac{1}{q} D D^{q-1} \Eulerop^q \gamma
=h_V \Eulerop^{\circ}\gamma +d_H h_H \gamma.
$$
Proposition~\ref{proposition:L_d_vanish} yields that~$\Img(d_H)\subset \Ker(\Eulerop)$.
If~$\gamma\in \Ker(\Eulerop)$, then the identity~\eqref{equation:homotopy_equality_variational} becomes~$\gamma=d_H (h_H \gamma)\in \Img(d_H)$. The exactness of the chain follows straightforwardly.
\end{proof}

We present an alternative horizontal homotopy operator on~$\Omega_0$ in Subsection~\ref{section:IBP}, with some examples of the outputs of both operators in Table~\ref{table:comparison_homotopy}.

\subsection{Exactness of the aromatic bicomplex}
\label{section:aromatic_bicomplex_theorem_proof}

This section is devoted to the proof of Theorem~\ref{theorem:exact_aro_var_bi}. We start by showing the aromatic bicomplex is indeed a bicomplex.
\begin{proof}[Proof of Proposition~\ref{proposition:derivatives_squared}]
Let~$\gamma\in\FF_{n,p}$, and~$(r_1,\dots,r_n)$ its roots. The horizontal derivative satisfies
\begin{align*}
d_H^2 \wedge \gamma
&=\frac{1}{n!}\wedge^{\circ} \sum_{\sigma\in \SS_{n}^{\bullet}} \varepsilon(\sigma) D^{r_{n-1}}D^{r_n} \sigma\gamma\\
&=\frac{1}{n!}\wedge^{\circ}\sum_{\sigma\in \SS_{n}^{\bullet}} \varepsilon((n(n-1))\sigma) D^{r_{n-1}}D^{r_n} (n(n-1))\sigma\gamma\\
&=-\frac{1}{n!}\wedge^{\circ}\sum_{\sigma\in \SS_{n}^{\bullet}} \varepsilon(\sigma) D^{r_n}D^{r_{n-1}} \sigma\gamma\\
&=-d_H^2 \wedge \gamma,
\end{align*}
where~$(n(n-1))\in \SS_{n}^{\bullet}$ is a transposition and where we used that~$D^{r_{n-1}}$ and~$D^{r_n}$ commute.
For the vertical derivative, a similar approach on~$\gamma\in\FF_{n,p-1}$ gives
\begin{align*}
d_V^2 \wedge \gamma
&=\sum_{\underset{v\neq w}{v, w\in V^{\bullet}}} \wedge(\wedge \gamma_{v\rightarrow \Circled{p}})_{w\rightarrow \Circled{p+1}}\\
&=\frac{1}{p!(p+1)!} \wedge^{\bullet} \sum_{\underset{v\neq w}{v, w\in V^{\bullet}}} \sum_{\sigma\in\SS_{p}^{\circ}} \sum_{\widetilde{\sigma}\in\SS_{p+1}^{\circ}} \varepsilon(\widetilde{\sigma}\sigma) \widetilde{\sigma}(\sigma (\gamma_{v\rightarrow \Circled{p}})_{w\rightarrow \Circled{p+1}})\\
&=\frac{1}{p!(p+1)!} \wedge^{\bullet} \sum_{\underset{v\neq w}{v, w\in V^{\bullet}}} \sum_{\sigma\in\SS_{p}^{\circ}} \sum_{\widetilde{\sigma}\in\SS_{p+1}^{\circ}} \varepsilon(\widetilde{\sigma}\sigma) \widetilde{\sigma}\sigma (\gamma_{v\rightarrow \Circled{p},w\rightarrow \Circled{p+1}})\\
&=\frac{1}{p!(p+1)!} \wedge^{\bullet} \sum_{\underset{v\neq w}{v, w\in V^{\bullet}}} \sum_{\sigma\in\SS_{p}^{\circ}} \sum_{\widetilde{\sigma}\in\SS_{p+1}^{\circ}} \varepsilon(\widetilde{\sigma}\sigma(\Circled{p}\Circled{p+1})) \widetilde{\sigma}\sigma (\gamma_{v\rightarrow \Circled{p+1},w\rightarrow \Circled{p}})\\
&=-\frac{1}{p!(p+1)!} \wedge^{\bullet} \sum_{\underset{v\neq w}{v, w\in V^{\bullet}}} \sum_{\sigma\in\SS_{p}^{\circ}} \sum_{\widetilde{\sigma}\in\SS_{p+1}^{\circ}} \varepsilon(\widetilde{\sigma}\sigma) \widetilde{\sigma}(\sigma (\gamma_{w\rightarrow \Circled{p}})_{v\rightarrow \Circled{p+1}})\\
&=-d_V^2\wedge \gamma.
\end{align*}
As~$\Omega_{n,p}=\wedge(\FF_{n,p})$, we get the desired identities. The commutativity of~$d_H$ and~$d_V$ is straightforward.
\end{proof}

In order to prove the exactness of the sequences appearing in the variational bicomplex, we rely on the use of two homotopy operators. We begin by the vertical homotopy as it is the easiest. Note that the variational homotopy operator and the vertical homotopy operator introduced in Theorem~\ref{theorem:exact_variational_chain} coincide for~$p=1$.
\begin{proposition}
\label{proposition:vertical_homotopy}
The vertical homotopy operator~$h_V\colon\Omega_{n,p}\rightarrow\Omega_{n,p-1}$ given by
$$h_V \gamma=\frac{p}{\abs{\gamma}}\gamma_{\Circled{p}\rightarrow\bullet},$$
satisfies for~$p\geq 1$ and~$\gamma\in \Omega_{n,p}$,
$$
(d_V h_V +h_V d_V)\gamma=\gamma.
$$
\end{proposition}

\begin{proof}
Let~$\gamma\in \FF_{n,p}$.
On the first hand, we have
\begin{align*}
d_V h_V \gamma
&=\frac{p}{\abs{\gamma}} \wedge \gamma
+\frac{p}{\abs{\gamma}}\sum_{v\in V^{\bullet}} \wedge \gamma_{\Circled{p}\rightarrow\bullet,v\rightarrow \Circled{p}}\\
&=\frac{p}{\abs{\gamma}}\wedge \gamma
+\frac{1}{\abs{\gamma}(p-1)!}\wedge^{\bullet} \sum_{v\in V^{\bullet}} \sum_{\sigma\in \SS_p^{\circ}} \varepsilon(\sigma) \sigma(\gamma_{\Circled{p}\rightarrow\bullet,v\rightarrow \Circled{p}})
.
\end{align*}
On the other hand, we have
\begin{align*}
h_V d_V\gamma
&=\frac{1}{\abs{\gamma}} \sum_{v\in V^{\bullet}} h_V(\wedge \gamma_{v\rightarrow \Circled{p+1}})\\
&=\frac{p+1}{\abs{\gamma}(p+1)!} \wedge^{\bullet} \sum_{v\in V^{\bullet}} \sum_{\sigma\in \SS_{p+1}^{\circ}} \varepsilon(\sigma) \sigma(\gamma_{v\rightarrow \Circled{p+1}})_{\Circled{p+1}\rightarrow \bullet}\\
&=\frac{1}{\abs{\gamma}p!} \wedge^{\bullet} \sum_{v\in V^{\bullet}} \sum_{\underset{\sigma(\Circled{p+1})=\Circled{p+1}}{\sigma\in \SS_{p+1}^{\circ}}} \varepsilon(\sigma) \sigma(\gamma_{v\rightarrow \Circled{p+1}})_{\Circled{p+1}\rightarrow \bullet}\\
&+\frac{1}{\abs{\gamma}p!} \wedge^{\bullet} \sum_{v\in V^{\bullet}} \sum_{\underset{\sigma(\Circled{p+1})\neq\Circled{p+1}}{\sigma\in \SS_{p+1}^{\circ}}} \varepsilon(\sigma) \sigma(\gamma_{v\rightarrow \Circled{p+1}})_{\Circled{p+1}\rightarrow \bullet}\\
&=\frac{1}{\abs{\gamma}p!} \wedge^{\bullet} \sum_{v\in V^{\bullet}} \sum_{\sigma\in \SS_{p}^{\circ}} \varepsilon(\sigma) \sigma\gamma\\
&+\frac{1}{\abs{\gamma}p!} \wedge^{\bullet} \sum_{i=1}^p \sum_{v\in V^{\bullet}} \sum_{\underset{\sigma(\Circled{i})=\Circled{p}}{\sigma\in \SS_{p}^{\circ}}} \varepsilon(\sigma) \sigma(\gamma_{\Circled{i}\rightarrow \bullet, v\rightarrow \Circled{p}})\\
&=\frac{\abs{\gamma}-p}{\abs{\gamma}} \wedge \gamma
+\frac{1}{\abs{\gamma}p!} \wedge^{\bullet} \sum_{i=1}^p \sum_{v\in V^{\bullet}} \sum_{\sigma\in \SS_{p}^{\circ}} \varepsilon(\sigma(\Circled{i}\Circled{p})) \sigma(\gamma_{\Circled{p}\rightarrow \bullet, v\rightarrow \Circled{p}})\\
&=\frac{\abs{\gamma}-p}{\abs{\gamma}} \wedge \gamma
-\frac{1}{\abs{\gamma}(p-1)!} \wedge^{\bullet} \sum_{v\in V^{\bullet}} \sum_{\sigma\in \SS_p^{\circ}} \varepsilon(\sigma) \sigma(\gamma_{\Circled{p}\rightarrow \bullet, v\rightarrow \Circled{p}}).
\end{align*}
We deduce that
$$(d_V h_V +h_V d_V)\gamma=\wedge \gamma,$$
and we obtain the desired equality by linearity as~$\Omega_{n,p}=\wedge \Span(\FF_{n,p})$ and~$\wedge^2=\wedge$.
\end{proof}

The expression of the horizontal homotopy operator is much more technical. We refer the reader to~\cite{Anderson80ote,Tulczyjew80tel,Olver93aol} for its derivation in the context of differential geometry.
Note that for~$n=0$, the horizontal homotopy operator coincides with the one introduced for the variational complex in Theorem~\ref{theorem:exact_variational_chain}.
\begin{proposition}
\label{proposition:horizontal_homotopy}
We define the horizontal homotopy operator~$h_H\colon\Omega_{n,p}\rightarrow\Omega_{n+1,p}$ by
$$h_H \gamma=\frac{1}{\abs{\gamma}} \sum_{q=0}^\infty \frac{n+1}{q+n+1} \wedge D^{q} \Eulerop^{q+1} \gamma.$$
It satisfies for~$n\geq 1$ and~$\gamma\in \Omega_{n,p}$,
$$
(d_H h_H+h_H d_H)\gamma=\gamma.
$$
\end{proposition}

\begin{proof}
For the sake of simplicity, we consider~$\gamma\in \FF_n$. We define for~$v\in V$
\begin{align*}
J_v \gamma
&=\sum_{q=0}^\infty \sum_{\underset{\abs{I}=q}{\Pi(v)=\{\hat{r}\}\cup I\cup J}} \frac{n+1}{q+n+1} D^I \Eulerop^{I,\hat{r}}_{v} \gamma,
\end{align*}
where we denote~$\Eulerop^{I,\hat{r}}_{v} \gamma=(-1)^{\abs{\Pi(v)}-(q+1)} (D^J \gamma_{v^{\starsymbol}})^{\starsymbol}$ and where~$\hat{r}$ becomes the new root of~$J_v \gamma$.
On the first hand, we have
$$
d_H \wedge\gamma= d_H \sum_{\sigma\in \SS_n} \frac{\varepsilon(\sigma)}{n!} \sigma \gamma
=\sum_{\sigma\in \SS_n} \frac{\varepsilon(\sigma)}{n!} D^{r_{\sigma^{-1}(n)}} \sigma \gamma,
$$
and~$J_v  d_H \wedge \gamma~$ is given by
\begin{align*}
J_v d_H \wedge\gamma
&=\sum_{q=0}^\infty \sum_{\sigma\in \SS_n} \sum_{u\neq v} \sum_{\underset{\abs{I}=q}{\Pi(v)=\{\hat{r}\}\cup I\cup J}} \frac{\varepsilon(\sigma)}{(q+n)(n-1)!} D^I \Eulerop^{I,\hat{r}}_{v} D^{r_{\sigma^{-1}(n)}\rightarrow u} \sigma \gamma\\
&+\sum_{q=0}^\infty \sum_{\sigma\in \SS_n} \sum_{\underset{\abs{I}=q}{\Pi(v)\cup \{r_k\}=\{\hat{r}\}\cup I\cup J}} \frac{\varepsilon(\sigma)}{(q+n)(n-1)!} D^I \Eulerop^{I,\hat{r}}_{v} D^{r_{\sigma^{-1}(n)}\rightarrow v} \sigma \gamma.
\end{align*}
The expression of~$D^I \Eulerop^{I,\hat{r}}_{v} D^{r\rightarrow u} \gamma$ satisfies for~$u\in V$:
$$
D^I \Eulerop^{I,\hat{r}}_{v} D^{r\rightarrow u} \gamma
=\left\{\begin{array}{l}
D^I \Eulerop^{I}_{v} \gamma \quad \text{if} \quad \hat{r}= r,\quad u=v,\\
\sum_{w\in V} D^{r\rightarrow w} D^{I\setminus\{r\}} \Eulerop^{I\setminus\{r\},\hat{r}}_{v} \gamma \quad \text{if} \quad r\in I,\quad u=v,\\
-\sum_{w\neq v} D^{r\rightarrow w} D^I \Eulerop^{I,\hat{r}}_{v} \gamma  \quad \text{if} \quad r\in J,\quad u=v,\\
D^{r\rightarrow u} D^I \Eulerop^{I,\hat{r}}_{v} \gamma \quad \text{if} \quad u\neq v.
\end{array}\right.
$$
Thus, for~$\gamma\in\FF_{n,p}$,~$J_v d_H \wedge \gamma~$ is given by
\begin{align*}
J_v d_H \wedge \gamma 
&=\sum_{q=0}^\infty \sum_{\sigma\in \SS_n} \sum_{\underset{\abs{I}=q}{\Pi(v)= I\cup J}} \frac{\varepsilon(\sigma)}{(q+n)(n-1)!} D^I \Eulerop^{I}_{v} \sigma \gamma\\
&+\sum_{q=1}^\infty \sum_{\sigma\in \SS_n} \sum_{w\in V} \sum_{\underset{\abs{I}=q-1}{\Pi(v)=\{\hat{r}\}\cup I\cup J}} \frac{\varepsilon(\sigma)}{(q+n)(n-1)!} D^{r_{\sigma^{-1}(n)}\rightarrow w} D^I \Eulerop^{I,\hat{r}}_{v} \sigma \gamma\\
&-\sum_{q=0}^\infty \sum_{\sigma\in \SS_n} \sum_{w\neq v} \sum_{\underset{\abs{I}=q}{\Pi(v)=\{\hat{r}\}\cup I\cup J}} \frac{\varepsilon(\sigma)}{(q+n)(n-1)!} D^{r_{\sigma^{-1}(n)}\rightarrow w} D^I \Eulerop^{I,\hat{r}}_{v} \sigma \gamma\\
&+\sum_{q=0}^\infty \sum_{\sigma\in \SS_n} \sum_{u\neq v} \sum_{\underset{\abs{I}=q}{\Pi(v)=\{\hat{r}\}\cup I\cup J}} \frac{\varepsilon(\sigma)}{(q+n)(n-1)!} D^{r_{\sigma^{-1}(n)}\rightarrow u} D^I \Eulerop^{I,\hat{r}}_{v} \sigma \gamma\\
&=\sum_{q=0}^\infty \frac{n}{q+n} D^q \Eulerop^{q}_{v} \wedge \gamma\\
&+\sum_{q=0}^\infty \sum_{\sigma\in \SS_n} \sum_{\underset{\abs{I}=q}{\Pi(v)=\{\hat{r}\}\cup I\cup J}} \frac{\varepsilon(\sigma)}{(q+n+1)(n-1)!} D^{r_{\sigma^{-1}(n)}} D^I \Eulerop^{I,\hat{r}}_{v} \sigma \gamma.
\end{align*}
As~$D^q \Eulerop^{q}_{v} \wedge \gamma$ is unchanged by the application of the wedge operator, we find
\begin{align*}
\wedge J_v d_H \wedge \gamma
&=\sum_{q=0}^\infty \frac{n}{q+n} D^q \Eulerop^{q}_{v} \wedge \gamma\\
&+\sum_{q=0}^\infty \sum_{\nu\in \SS_n} \sum_{\sigma\in \SS_n} \sum_{\underset{\abs{I}=q}{\Pi(v)=\{\hat{r}\}\cup I\cup J}}  \frac{\varepsilon(\sigma\nu)}{(q+n+1)(n-1)!n!} \nu D^{r_{\sigma^{-1}(n)}} D^I \Eulerop^{I,\hat{r}}_{v} \sigma \gamma.
\end{align*}
On the other hand, we have
\begin{align*}
d_H & \wedge J_v \wedge \gamma
=\sum_{\sigma\in \SS_n} \frac{\varepsilon(\sigma)}{n!} d_H \wedge J_v \sigma \gamma\\
&=d_H\sum_{q=0}^\infty \sum_{\sigma\in \SS_n} \sum_{\nu\in \SS_{n+1}}\sum_{\underset{\abs{I}=q}{\Pi(v)=\{\hat{r}\}\cup I\cup J}}  \frac{\varepsilon(\sigma\nu)}{(q+n+1)(n!)^2} \nu D^I \Eulerop^{I,\hat{r}}_{v} \sigma \gamma\\
&=\sum_{q=0}^\infty \sum_{\sigma\in \SS_n}\sum_{\nu\in \SS_{n+1}} \sum_{\underset{\abs{I}=q}{\Pi(v)=\{\hat{r}\}\cup I\cup J}}  \frac{\varepsilon(\sigma\nu)}{(q+n+1)(n!)^2} D^{r_{(\sigma\nu)^{-1}(n+1)}} \nu D^I \Eulerop^{I,\hat{r}}_{v} \sigma \gamma\\
&=\sum_{q=0}^\infty \sum_{\sigma\in \SS_n} \sum_{\underset{\abs{I}=q}{\Pi(v)=\{\hat{r}\}\cup I\cup J}}  \frac{\varepsilon(\sigma)}{(q+n+1)n!} D^{I,\hat{r}} \Eulerop^{I,\hat{r}}_{v} \sigma \gamma\\
&-\sum_{q=0}^\infty \sum_{\sigma\in \SS_n}\sum_{\nu\in \SS_n}\sum_{\underset{\eta(n)=n}{\eta\in \SS_n}} \sum_{\underset{\abs{I}=q}{\Pi(v)=\{\hat{r}\}\cup I\cup J}}  \frac{\varepsilon(\sigma\nu\eta)}{(q+n+1)(n-1)!(n!)^2} D^{r_{(\sigma\nu)^{-1}(n)}} \eta (n(n+1)) \nu D^I \Eulerop^{I,\hat{r}}_{v} \sigma \gamma\\
&=\sum_{q=0}^\infty \frac{q+1}{q+n+1} D^{q+1} \Eulerop^{q+1}_{v} \wedge \gamma\\
&-\sum_{q=0}^\infty \sum_{\sigma\in \SS_n} \sum_{\eta\in \SS_n} \sum_{\underset{\abs{I}=q}{\Pi(v)=\{\hat{r}\}\cup I\cup J}}  \frac{\varepsilon(\sigma\eta)}{(q+n+1)(n-1)!n!} \eta D^{r_{\sigma^{-1}(n)}} D^I \Eulerop^{I,\hat{r}}_{v} \sigma \gamma\\
&=\sum_{q=1}^\infty \frac{q}{q+n} D^q \Eulerop^q_{v} \wedge \gamma\\
&-\sum_{q=0}^\infty \sum_{\sigma\in \SS_n} \sum_{\eta\in \SS_n} \sum_{\underset{\abs{I}=q}{\Pi(v)=\{\hat{r}\}\cup I\cup J}}  \frac{\varepsilon(\sigma\eta)}{(q+n+1)(n-1)!n!} \eta D^{r_{\sigma^{-1}(n)}} D^I \Eulerop^{I,\hat{r}}_{v} \sigma \gamma,
\end{align*}
where we split the cases~$\sigma\nu(n+1)=n+1$ and~$\sigma\nu(n+1)\neq n+1$, and where we substitute~$\nu$ by~$\nu(n(n+1))\eta$ in the latter case. The last equality is obtained by substituting~$\sigma$ with~$\nu\sigma$.
Using Proposition~\ref{prop:higher_Euler_operators_Omega}, we get
\begin{equation}
\label{equation:horizontal_homotopy_labeled}
(\wedge J_v d_H
+d_H \wedge J_v) \wedge \gamma
=\sum_{q=0}^\infty D^q \Eulerop^{q}_{v} \wedge \gamma=\wedge \gamma.
\end{equation}
Summing on~$v\in V$ gives the desired homotopy identity.
\end{proof}

\begin{remark}
The identity~\eqref{equation:horizontal_homotopy_labeled} suggests simpler homotopy operators, obtained by fixing a node~$v$ and considering the operator~$\wedge J_v$. We emphasize that this approach does not work. Indeed, given an aromatic form~$\gamma\in \Omega_{n,p}$, there is no canonical choice of the node~$v$, so that~$J_v \gamma$ is ill-defined, where~$h_V \gamma$ is well-defined. In the proof of Proposition~\ref{proposition:horizontal_homotopy},~$J_v \gamma$ makes sense since we consider a single aromatic forest~$\gamma\in \FF_{n,p}$.
\end{remark}

\subsection{The aromatic bicomplex with a divergence-free vector field}
\label{section:divergence-free_bicomplex}

This section is devoted to the proof of Theorem~\ref{theorem:exact_div_free_complex}.
Proposition~\ref{proposition:derivatives_squared} and Proposition~\ref{proposition:vertical_homotopy} extend naturally to the divergence-free context. Proposition~\ref{proposition:horizontal_homotopy} is not valid anymore and we replace it by the following result.
\begin{proposition}
\label{proposition:divergence-free_horizontal_homotopy}
We define
$$\widetilde{h}_H \gamma=\frac{1}{\abs{\gamma}} \sum_{v\in V} \sum_{q=0}^\infty \frac{n+1}{q+n+\ind_{v\notin R}} \wedge D^{q} \Eulerop^{q+1}_v \gamma.$$
For~$n>1$ and~$\gamma\in \widetilde{\Omega}_{n,p}$, the horizontal homotopy identity is
$$
(d_H \widetilde{h}_H+\widetilde{h}_H d_H)\gamma=\gamma,
$$
while for~$n=1$, the identity is
\begin{equation}
\label{equation:homotopy_horizontal_div_free_n=1}
(d_H \widetilde{h}_H+\widetilde{h}_H d_H)\gamma=\gamma-R\gamma.
\end{equation}
The remainder in~\eqref{equation:homotopy_horizontal_div_free_n=1} is the linear map~$R\gamma=\frac{1}{\abs{\gamma}}\Eulerop_r \gamma$ for~$\gamma\in \widetilde{\FF}_{1,p}$, with~$r$ the root of~$\gamma$.
Moreover, if~$\gamma\in \widetilde{\Omega}_{1,p}^N$ satisfies~$d_H \gamma=0$, then~$R\gamma=0$ if and only if~$N>1$.
In particular, the horizontal sequences in the divergence-free aromatic bicomplex are exact.
\end{proposition}


The proof of Proposition~\ref{proposition:divergence-free_horizontal_homotopy} follows the structure and notations of the proof of Proposition~\ref{proposition:horizontal_homotopy}. We recall that the notations in the proof follow the convention of Remark~\ref{remark:clarification_notation}.
\begin{proof}
We consider for simplicity~$\gamma\in \widetilde{\FF}_n$.
For~$v\in V$, we define
\begin{align*}
\widetilde{J}_v(\gamma)
&=\sum_{q=0}^\infty \sum_{\underset{\abs{I}=q}{\Pi(v)=\{\hat{r}\}\cup I\cup J}} \frac{n+1}{q+n+\ind_{v\notin R}} D^I \Eulerop^{I,\hat{r}}_{v} \gamma.
\end{align*}
We have if~$r=v$,
$$
D^I \Eulerop^{I,\hat{r}}_{v} D^{r\rightarrow u} \gamma
=\left\{\begin{array}{l}
0 \quad \text{if} \quad u=v,\\
D^{r\rightarrow u} D^I \Eulerop^{I,\hat{r}}_{v} \gamma \quad \text{if} \quad u\neq v,
\end{array}\right.
$$
and if~$r\neq v$,
$$
D^I \Eulerop^{I,\hat{r}}_{v} D^{r\rightarrow u} \gamma
=\left\{\begin{array}{l}
D^I \Eulerop^{I}_{v} \gamma \quad \text{if} \quad \hat{r}= r,\quad u=v,\\
\sum_{w\in V} D^{r\rightarrow w} D^{I\setminus\{r\}} \Eulerop^{I\setminus\{r\},\hat{r}}_{v} \gamma \quad \text{if} \quad r\in I,\quad u=v,\\
-\sum_{w\neq v} D^{r\rightarrow w} D^I \Eulerop^{I,\hat{r}}_{v} \gamma  \quad \text{if} \quad r\in J,\quad u=v,\\
D^{r\rightarrow u} D^I \Eulerop^{I,\hat{r}}_{v} \gamma \quad \text{if} \quad u\neq v.
\end{array}\right.
$$
Thus,~$\widetilde{J}_v d_H \wedge \gamma$ is given by
\begin{align*}
\widetilde{J}_v d_H \wedge \gamma
&=\ind_{v\neq r_{\sigma^{-1}(n)}}\Big[\sum_{q=0}^\infty \sum_{\sigma\in \SS_n} \sum_{\underset{\abs{I}=q}{\Pi(v)= I\cup J}} \frac{\varepsilon(\sigma)}{(q+n-\ind_{v\in R})(n-1)!} D^I \Eulerop^{I}_{v} \sigma \gamma\\
&+\sum_{q=1}^\infty \sum_{\sigma\in \SS_n} \sum_{w\in V} \sum_{\underset{\abs{I}=q-1}{\Pi(v)=\{\hat{r}\}\cup I\cup J}} \frac{\varepsilon(\sigma)}{(q+n-\ind_{v\in R})(n-1)!} D^{r_{\sigma^{-1}(n)}\rightarrow w} D^I \Eulerop^{I,\hat{r}}_{v} \sigma \gamma\\
&-\sum_{q=0}^\infty \sum_{\sigma\in \SS_n} \sum_{w\neq v} \sum_{\underset{\abs{I}=q}{\Pi(v)=\{\hat{r}\}\cup I\cup J}} \frac{\varepsilon(\sigma)}{(q+n-\ind_{v\in R})(n-1)!} D^{r_{\sigma^{-1}(n)}\rightarrow w} D^I \Eulerop^{I,\hat{r}}_{v} \sigma \gamma\Big]\\
&+\sum_{q=0}^\infty \sum_{\sigma\in \SS_n} \sum_{u\neq v} \sum_{\underset{\abs{I}=q}{\Pi(v)=\{\hat{r}\}\cup I\cup J}} \frac{\varepsilon(\sigma)}{(q+n-\ind_{v\in R})(n-1)!} D^{r_{\sigma^{-1}(n)}\rightarrow u} D^I \Eulerop^{I,\hat{r}}_{v} \sigma \gamma\\
&=\ind_{v\neq r_{\sigma^{-1}(n)}} \sum_{q=0}^\infty \sum_{\sigma\in \SS_n}  \frac{\varepsilon(\sigma)}{(q+n-\ind_{v\in R})(n-1)!} D^q \Eulerop^{q}_{v} \sigma \gamma\\
&+\sum_{q=0}^\infty \sum_{\sigma\in \SS_n} \sum_{\underset{\abs{I}=q}{\Pi(v)=\{\hat{r}\}\cup I\cup J}} \frac{\varepsilon(\sigma)}{(q+n+\ind_{v\notin R})(n-1)!} D^{r_{\sigma^{-1}(n)}} D^I \Eulerop^{I,\hat{r}}_{v} \sigma \gamma
\end{align*}
Then, we find for~$n=1$ and~$v\in R$,
\begin{align*}
\wedge \widetilde{J}_v d_H \wedge \gamma
&=\sum_{q=0}^\infty \sum_{\sigma\in \SS_n} \sum_{\nu\in \SS_n} \sum_{\underset{\abs{I}=q}{\Pi(v)=\{\hat{r}\}\cup I\cup J}} \frac{\varepsilon(\sigma\nu)}{(q+n+\ind_{v\notin R})(n-1)!n!} \nu D^{r_{\sigma^{-1}(n)}} D^I \Eulerop^{I,\hat{r}}_{v} \sigma \gamma,
\end{align*}
and otherwise
\begin{align*}
\wedge \widetilde{J}_v d_H \wedge \gamma
&=\sum_{q=0}^\infty \frac{n-\ind_{v\in R}}{q+n-\ind_{v\in R}} D^q \Eulerop^{q}_{v} \wedge \gamma\\
&+\sum_{q=0}^\infty \sum_{\sigma\in \SS_n} \sum_{\nu\in \SS_n} \sum_{\underset{\abs{I}=q}{\Pi(v)=\{\hat{r}\}\cup I\cup J}} \frac{\varepsilon(\sigma\nu)}{(q+n+\ind_{v\notin R})(n-1)!n!} \nu D^{r_{\sigma^{-1}(n)}} D^I \Eulerop^{I,\hat{r}}_{v} \sigma \gamma,
\end{align*}
where for~$v\in R$, the wedge adds the missing term~$v= r_{\sigma^{-1}(n)}$, and rescales the expression with a coefficient~$\frac{n-1}{n}$.
On the other hand, we have
\begin{align*}
d_H \wedge \widetilde{J}_v \wedge \gamma
&=d_H \sum_{q=0}^\infty \sum_{\sigma\in \SS_n} \sum_{\nu\in \SS_{n+1}}\sum_{\underset{\abs{I}=q}{\Pi(v)=\{\hat{r}\}\cup I\cup J}}  \frac{\varepsilon(\sigma\nu)}{(q+n+\ind_{v\notin R})(n!)^2} \nu D^I \Eulerop^{I,\hat{r}}_{v} \sigma \gamma\\
&=\sum_{q=0}^\infty \sum_{\sigma\in \SS_n}\sum_{\nu\in \SS_{n+1}} \sum_{\underset{\abs{I}=q}{\Pi(v)=\{\hat{r}\}\cup I\cup J}}  \frac{\varepsilon(\sigma\nu)}{(q+n+\ind_{v\notin R})(n!)^2} D^{r_{(\sigma\nu)^{-1}(n+1)}} \nu D^I \Eulerop^{I,\hat{r}}_{v} \sigma \gamma\\
&=\sum_{q=0}^\infty \sum_{\sigma\in \SS_n} \sum_{\underset{\abs{I}=q}{\Pi(v)=\{\hat{r}\}\cup I\cup J}}  \frac{\varepsilon(\sigma)}{(q+n+\ind_{v\notin R})n!} D^{I,\hat{r}} \Eulerop^{I,\hat{r}}_{v} \sigma \gamma\\
&-\sum_{q=0}^\infty \sum_{\sigma\in \SS_n}\sum_{\nu\in \SS_n}\sum_{\underset{\eta(n)=n}{\eta\in \SS_n}} \sum_{\underset{\abs{I}=q}{\Pi(v)=\{\hat{r}\}\cup I\cup J}}  \frac{\varepsilon(\sigma\nu\eta)}{(q+n+\ind_{v\notin R})(n-1)!(n!)^2} \\&\cdot D^{r_{(\sigma\nu)^{-1}(n)}} \eta (n(n+1)) \nu D^I \Eulerop^{I,\hat{r}}_{v} \sigma \gamma\\
&=\sum_{q=0}^\infty \frac{q+1}{q+n+\ind_{v\notin R}} D^{q+1} \Eulerop^{q+1}_{v} \wedge \gamma\\
&-\sum_{q=0}^\infty \sum_{\sigma\in \SS_n} \sum_{\eta\in \SS_n} \sum_{\underset{\abs{I}=q}{\Pi(v)=\{\hat{r}\}\cup I\cup J}}  \frac{\varepsilon(\sigma\eta)}{(q+n+\ind_{v\notin R})(n-1)!n!} \eta D^{r_{\sigma^{-1}(n)}} D^I \Eulerop^{I,\hat{r}}_{v} \sigma \gamma\\
&=\sum_{q=1}^\infty \frac{q}{q+n-\ind_{v\in R}} D^q \Eulerop^q_{v} \wedge \gamma\\
&-\sum_{q=0}^\infty \sum_{\sigma\in \SS_n} \sum_{\eta\in \SS_n} \sum_{\underset{\abs{I}=q}{\Pi(v)=\{\hat{r}\}\cup I\cup J}}  \frac{\varepsilon(\sigma\eta)}{(q+n+\ind_{v\notin R})(n-1)!n!} \eta D^{r_{\sigma^{-1}(n)}} D^I \Eulerop^{I,\hat{r}}_{v} \sigma \gamma.
\end{align*}
Using Proposition~\ref{prop:higher_Euler_operators_Omega}, we get for~$n>1$,
$$
(\wedge \widetilde{J}_v d_H +d_H \wedge \widetilde{J}_v) \wedge \gamma
=\sum_{q=0}^\infty D^q \Eulerop^{q}_{v} \wedge \gamma=\wedge \gamma,
$$
and for~$n=1$,
$$
(\wedge \widetilde{J}_v d_H +d_H \wedge \widetilde{J}_v) \wedge \gamma
=\sum_{q=0}^\infty D^q \Eulerop^{q}_{v} \wedge \gamma - \Eulerop_{r} \wedge \gamma \ind_{v\in R}=\wedge \gamma- \Eulerop_{r} \wedge \gamma \ind_{v\in R}.
$$
Summing on~$v\in V$ gives the desired homotopy identities.

Following the proof of Proposition~\ref{proposition:L_d_vanish}, we observe that~$\Eulerop_r d_H=0$ on~$\Omega_{2,p}$. Thus, we deduce from the identity~\eqref{equation:homotopy_horizontal_div_free_n=1} that~$\gamma\in \Img(d_H)$ if and only if~$d_H \gamma=0$ and~$\Eulerop_r \gamma=0$ in the case~$n=1$.
Assume that~$d_H\gamma=0$ for~$\gamma\in\Omega_{1,p}^N$, and apply~$\Eulerop_r$ to the homotopy identity. As~$\Eulerop_r^2=\Eulerop_r$, it yields
$$
\Eulerop_r\gamma=\frac{1}{N}\Eulerop_r \gamma.
$$
We deduce that~$\Eulerop_r \gamma=0$ or~$N=1$. Thus, the bicomplex of order~$N$ is exact if and only if~$N\neq 1$.
\end{proof}

\begin{remark}
Let~$\gamma\in \widetilde{\Omega}_1^N$ with~$N>1$, define the modified homotopy operators by
$$\widetilde{h}_H^1=\widetilde{h}_H\Big(1+\frac{1}{N-1}\Eulerop\Big), \quad \widetilde{h}_H^2=\widetilde{h}_H\Big(1+\frac{1}{N-1}\Eulerop_r\Big).$$
Then, the homotopy identity~\eqref{equation:homotopy_horizontal_div_free_n=1} on~$\widetilde{\Omega}_1^N$ with~$N>1$ is replaced by the simpler identity
\begin{equation}
\label{equation:simpler_div_free_homotopy}
(d_H \widetilde{h}_H^2 + \widetilde{h}_H^1 d_H )\gamma=\gamma.
\end{equation}
The proof of~\eqref{equation:simpler_div_free_homotopy} relies on the identities~$\Eulerop_r^2=\Eulerop_r$ and~$\Eulerop d_H=d_H \Eulerop_r$ on~$\widetilde{\Omega}_1$.
\end{remark}

\section{Applications and extensions}
\label{section:applications}

The study of the aromatic bicomplex brings a variety of new theoretical results, as presented in Section~\ref{section:preliminaries}, but also direct applications in numerical analysis.
Subsection~\ref{section:augmented_bicomplex} is devoted to the study of the generalised aromatic bicomplex, a natural extension of the aromatic bicomplex that includes the Euler-Lagrange complex.
In Subsection~\ref{section:combinatorics_Robert}, we deduce from the exactness of the aromatic bicomplex the dimensions of the spaces in the first two rows of the bicomplex, as well, as the number of solenoidal forms.
We describe further the divergences and the solenoidal forms in Subsection~\ref{section:description_image_kernel_dH}.
In Subsection~\ref{section:IBP}, we draw a bridge between the existing notions of integration by parts of Butcher trees by defining an alternative horizontal homotopy operator.
In Subsection~\ref{section:application_vp_integrators}, we give an explicit description of the B-series of an aromatic volume-preserving integrator and we prove that an aromatic Runge-Kutta method cannot be volume-preserving.

\subsection{The augmented aromatic bicomplex}
\label{section:augmented_bicomplex}

Following~\cite{Anderson89tvb,Anderson92itt} and the work on the variational complex~\eqref{equation:variational_complex} of Subsection~\ref{section:Euler_operators}, we augment the aromatic bicomplex with the aromatic equivalent of the Euler-Lagrange complex.

For~$p\geq 1$, define the interior Euler operator~$I\colon \Omega_{0,p} \rightarrow \Omega_{0,p}$ by
$$
I\gamma=\wedge \Eulerop_{\Circled{p}} \gamma= (-1)^{\abs{\Pi(\Circled{p})}} \wedge (D^{\abs{\Pi(\Circled{p})}} \gamma_{\Circled{p}^{\starsymbol}})^{\starsymbol},
$$
write~$\II_p=I(\Omega_{0,p})$,~$\II_p^N=I(\Omega_{0,p}^N)$, and the variational derivative~$\delta_V=I\circ d_V$.
The augmented aromatic bicomplex is drawn in Figure~\ref{figure:augmented_bicomplex}.
\begin{figure}[!ht]
$$\begin{tikzcd}
 & \vdots & \vdots & \vdots & \vdots & \\
    \dots \arrow{r}{d_H} & \Omega_{2,2} \arrow{r}{d_H} \arrow{u}{d_V} & \Omega_{1,2} \arrow{r}{d_H} \arrow{u}{d_V} & \Omega_{0,2} \arrow{u}{d_V} \arrow{r}{I} & \II_2 \arrow{u}{\delta_V} \arrow{r} & 0\\
    \dots \arrow{r}{d_H} & \Omega_{2,1} \arrow{r}{d_H} \arrow{u}{d_V} & \Omega_{1,1} \arrow{r}{d_H} \arrow{u}{d_V} & \Omega_{0,1} \arrow{u}{d_V} \arrow{r}{I} & \II_1 \arrow{u}{\delta_V} \arrow{r} & 0\\
    \dots \arrow{r}{d_H} & \Omega_{2} \arrow{r}{d_H} \arrow{u}{d_V} & \Omega_{1} \arrow{r}{d_H} \arrow{u}{d_V} & \Omega_{0} \arrow{u}{d_V} \arrow{ru}{\delta_V} & & \\
     & 0 \arrow{u} & 0 \arrow{u} & 0 \arrow{u} & &
\end{tikzcd}$$
\caption{The augmented aromatic bicomplex.}
	\label{figure:augmented_bicomplex}
\end{figure}
The edge complex~\eqref{equation:Euler_Lagrange_complex} is called the Euler Lagrange complex, and it is the object of ultimate interest here. Note that the variational complex~\eqref{equation:variational_complex} is a subcomplex of the Euler Lagrange complex as~$\delta_V=\Eulerop^{\circ}$ on~$\Omega_0$.
\begin{equation}
\label{equation:Euler_Lagrange_complex}
\begin{tikzcd}
\dots \arrow{r}{d_H} & \Omega_{2} \arrow{r}{d_H} & \Omega_{1} \arrow{r}{d_H} & \Omega_{0} \arrow{r}{\delta_V} & \II_1 \arrow{r}{\delta_V}& \II_2 \arrow{r}{\delta_V}& \dots
\end{tikzcd}
\end{equation}

\begin{theorem}
\label{theorem:exact_generalized_bicomplex}
Define the augmented homotopy operators as
$$\mathfrak{h}_H\gamma=\wedge \sum_{q=1}^\infty \frac{1}{q} D^{q-1} \Eulerop^q_{\Circled{p}} \gamma,\quad \mathfrak{h}_V \gamma=I\circ h_V.$$
Then, the maps~$I$ and~$\delta_V$ satisfy
$$I^2=I,\quad I d_H=0,\quad \delta_V^2=0,
$$
and the following identities hold
\begin{align*}
\gamma&=(I+d_H \mathfrak{h}_H) \gamma, \quad \gamma \in \Omega_{0,p}, \quad p\geq 1,\\
\gamma&=(\delta_V  h_V+\mathfrak{h}_V  \delta_V) \gamma, \quad \gamma \in \II_1,\\
\gamma&=(\delta_V  \mathfrak{h}_V+\mathfrak{h}_V  \delta_V) \gamma, \quad \gamma \in \II_p, \quad p>1.
\end{align*}
In particular, the horizontal and vertical sequences of the augmented aromatic bicomplex are exact, and the Euler-Lagrange complex~\eqref{equation:Euler_Lagrange_complex} is exact.
\end{theorem}

We follow the approach of \cite[Chap.\ts 4]{Anderson89tvb} for the proof of Theorem~\ref{theorem:exact_generalized_bicomplex}.
\begin{proof}
We first observe that for all~$\gamma\in\FF_{1,p}$,
$$
I d_H \gamma=\wedge \sum_{v\neq \Circled{p}} (-1)^{\abs{\Pi(\Circled{p})}} (D^{\abs{\Pi(\Circled{p})}} D^{r_1\rightarrow v} \gamma_{\Circled{p}^{\starsymbol}})^{\starsymbol}
+ \wedge(-1)^{\abs{\Pi(\Circled{p})}+1} (D^{\Pi(\Circled{p})}D^{r_1} \gamma_{\Circled{p}^{\starsymbol}})^{\starsymbol}=0.
$$
Using Equation~\eqref{equation:decomposition_Euler_op_v} with~$v=\Circled{p}$ gives the augmented horizontal homotopy identity
$$
(I+d_H \mathfrak{h}_H)\wedge \gamma=\wedge\gamma.
$$
Applying~$I$ to this last equality yields~$I\circ I=I$.
Let us now look at the Euler-Lagrange complex.
Let~$\gamma\in\Omega_{0,p}$, then~$d_V \gamma\in\Omega_{0,p+1}$. We apply the augmented horizontal homotopy identity to~$d_V \gamma$,
$$(I d_V +d_H \mathfrak{h}_H d_V)\gamma=d_V \gamma.$$
We apply~$d_V$ and use that~$d_V$ and~$d_H$ commute:
$$(d_V I d_V +d_H d_V \mathfrak{h}_H d_V)\gamma=0.$$
Since~$I d_H=0$, applying~$I$ yields 
$$\delta_V^2 \gamma=0.$$
Let~$\gamma\in \II_p$, the augmented horizontal homotopy applied to~$d_V\gamma$ gives
\begin{equation}
\label{equation:proof_augmented_homotopy_1}
d_V\gamma=\delta_V \gamma+d_H \mathfrak{h}_H d_V\gamma.
\end{equation}
If~$p=1$, then we use the identity~\eqref{equation:proof_augmented_homotopy_1} in the vertical homotopy identity to get
$$
\gamma=d_V h_V \gamma+h_V \delta_V \gamma +d_H \widetilde{\gamma},
$$
where we used that~$h_V$ and~$d_H$ commute and where~$\widetilde{\gamma}\in \Omega_{1,1}$.
Applying~$I$ yields
$$
\gamma=(\delta_V h_V+\mathfrak{h}_V \delta_V) \gamma.
$$
If~$p> 1$, we apply the augmented horizontal homotopy identity to~$h_V\gamma$,
$$h_V \gamma=\mathfrak{h}_V \gamma+d_H \mathfrak{h}_H h_V \gamma.$$
Applying the vertical derivative~$d_V$ gives
\begin{equation}
\label{equation:proof_augmented_homotopy_2}
d_V h_V \gamma=d_V \mathfrak{h}_V\gamma+d_H d_V \mathfrak{h}_H h_V \gamma,
\end{equation}
where we used that~$d_H$ and~$d_V$ commute.
We use the identities~\eqref{equation:proof_augmented_homotopy_1} and~\eqref{equation:proof_augmented_homotopy_2} in the vertical homotopy identity to get
$$
\gamma=d_V \mathfrak{h}_V \gamma+h_V  \delta_V \gamma +d_H \widetilde{\gamma},
$$
where~$\widetilde{\gamma}\in \Omega_{1,p}$.
As~$I\circ d_H=0$ and~$I \gamma=\gamma$, we find
$$
\gamma=(\delta_V \mathfrak{h}_V +\mathfrak{h}_V \delta_V) \gamma.
$$
The exactness of the augmented aromatic bicomplex is a straighforward consequence of Theorem~\ref{theorem:exact_aro_var_bi} and of the augmented homotopy identities.
\end{proof}

\subsection{Combinatorics on the aromatic bicomplex}
\label{section:combinatorics_Robert}

In this subsection, we determine the dimensions of the bottom two rows of the augmented aromatic bicomplex in the standard and in the divergence-free case. The primary motivation was to compute the dimension of the space of solenoidal (i.e.,  divergence-free) aromatic trees of each order. However, the result revealed a surprisingly simple connection with another combinatorial object, the self-looped scalar aromas, which allowed the construction of the fundamental spaces associated with the divergence operator.

We recall the fundamental generating functions associated with graphical enumeration. The number~$\TT^N$  of rooted trees of order~$N$ (sequence A000081 in the OEIS~\cite{OEIS}) has generating function
$$ t(z) = \sum_{N=1}^\infty \TT^N z^n = z + z^2 + 2 z^3 + 4 z^4 + 9 z^5 + 20 z^6 + \dots$$
and satisfies the functional equation
\begin{equation}
\label{eq:tid}
 t(z) = z \exp\left( \sum_{k=1}^\infty \frac{1}{k} t(z^k)\right).
 \end{equation}
Considering a rooted tree as a directed graph, each node except the root has a single outgoing edge. Thus, rooted trees are equivalent to the class of ``mapping patterns'' of functions~$\{1,\dots,N-1\}\to\{0,\dots,N-1\}$ modulo the symmetric group~$\SS_{N-1}$ (node 0 is the root, and we ``forget the labels'').

The number~$|\Omega^N_{0}|$ of scalar aromas in addition to the empty aroma (sequence A001372) has generating function
$$ a(z) := 1+\sum_{N=1}^\infty |\Omega^N_{0}| z^n = 1+ z + 3 z^2 + 7 z^3 + 19 z^4 + 47 z^5 + 130 z^6 + \dots$$
and is related to~$t(z)$ by the equation
$$ a(z) = \prod_{k=1}^\infty \left(1-t(z^k)\right)^{-1}.$$
The scalar aromas are mapping patterns of functions~$\{1,\dots,N\}\to\{1,\dots,N\}$ mod~$\SS_{N}$.

We introduce two new spaces, the aromatic forms with a~$1$-loop, called the self-looped aromatic forms~$\mathring{\Omega}_{n,k}$, and their complements, the non-self-looped aromatic forms~$\overline{\Omega}_{n,k}$, and generating functions~$\mathring{a}(z)$ of~$\mathring{\Omega}_{0}$ and~$\overline{a}(z)$ of~$\overline{\Omega}_{0}$.
As for mapping patterns on~$\{1,\dots,N\}$, the mappings in~$\mathring{\Omega}_{0}$ (enumerated by sequence A217896) have at least one fixed point  and those in~$\overline{\Omega}_{0}$ (also known as the ``functional digraphs'', enumerated by sequence A001373) have no fixed points.

Each pair consisting of one non-self-looped scalar aroma and one rooted tree generates a scalar aroma of one lower degree: cut off the root and replace each new root by a self-loop. The process is invertible (redirect all self-loops in an arbitrary scalar aroma to a new root). Expressed in terms of generating functions, it writes
$$ z a(z) = t(z) \overline{a}(z).$$
Therefore
\begin{equation}
\label{eq:aring}
 \mathring{a}(z) = a(z)-1-\overline{a}(z) = \frac{a(z)(t(z)-z)}{t(z)}-1 = z +2 z^2 + 5 z^3 + 13 z^4 + 34 z^5 + 90 z^6 + \dots.
\end{equation}

The following result enumerates the first two rows of the augmented bicomplex and the solenoidal forms. We refer to Table~\ref{table:dimensions_solenoidal} and Table~\ref{tab:dims} for the dimensions for the first orders~$N$.
\begin{theorem}
\label{theorem:counting_standard}
Let 
$$b_p(u,z) = \sum_{k=0}^\infty \sum_{N=1}^\infty \left| \Omega^N_{k,p}\right| u^k z^N$$
be the bivariate generating function for row~$p$ of the aromatic bicomplex, let
$$ c_p(z) = \sum_{N=1}^\infty \left| \mathcal{I}_p^N\right| z^N$$
be the generating function of the type-$p$ functional forms, and let
$$s(z) = \sum_{N=1}^\infty \left| \Psi^N\right| z^N$$
be the generating function of the solenoidal aromatic trees.
Then
\begin{align*}
b_0(u,z)&= a(z) \exp\left( \sum_{k=1}^\infty \frac{(-1)^{k-1}}{k}u^k t(z^k)\right),\\
b_1(u,z) &= b_0(u,z) \frac{t(z)(1 + u - u t(z))}{(1-t(z))^2},\\
c_1(z) &= z b_1(0,z) =  \frac{z a(z) t(z)}{(1-t(z))^2},\\
s(z) &= a(z)t(z)-\mathring{a}(z).
\end{align*}
\end{theorem}

\begin{proof}
We first note that the aromatic trees~$\Omega_{1}$ are given by the product of a scalar aroma (enumerated by~$a(z)$) with a rooted tree
(enumerated by~$t(z)$). This is the {\em Cartesian product} construction of enumerative combinatorics. Therefore the aromatic trees are enumerated\footnote{This is a right-shift of sequence A126285, the partial mapping patterns. From a partial mapping pattern on~$N-1$ nodes, add a new node (the root of the tree) and point any nodes without outgoing edges to it to get an aromatic vector field; to invert, cut off the root.} by~$a(z)t(z)$.

An element of~$\Omega^N_{k}$ is given by the product of a scalar (enumerated by~$a(z)$) with a wedge product of~$k$ rooted trees, enumerated by~$t(z)$. Since the rooted trees are unordered and distinct, this is the {\em power set} construction; the expression for~$b_0(u,z)$ then follows from \cite[Proposition III.5]{Flajolet09ac}.

The bottom row of the bicomplex is exact (Theorem~\ref{theorem:exact_aro_var_bi}), so the dimension of~$\Psi$ is given by the alternating sum
\begin{align*}
s(z) &= \sum_{N=1}^\infty \left(| \Omega_{2}^N | -  | \Omega_{3}^N | +  | \Omega_{4}^N | -  | \Omega_{5}^N | +\dots\right) z^N \\
&= \sum_{N=1}^\infty \left(|\Omega_{1}^N| - |\Omega_{0}^N|+\sum_{k=0}^\infty (-1)^k |\Omega_{k}^N| \right) z^N \\
&= a(z) t(z) - a(z)+1 +b_0(-1,z)\\
&= a(z) t(z) - a(z)+1 + a(z) \exp\left(-\sum_{k=1}^\infty \frac{1}{k} t(z^k))\right)  \\
&= a(z) t(z) - a(z)+1 + \frac{a(z)z}{t(z)} && \text{using Eq.~\eqref{eq:tid}}\\
&= a(z)t(z)-\mathring{a}(z) &&\text{using Eq.~\eqref{eq:aring}}.
\end{align*}

An element of~$\Omega_{n,1}$ is obtained from an element of~$\Omega_{n}$ by marking a node with the symbol~$\Circled{1}$. The marked node can be in one of the scalar aroma components or in one of the rooted tree components. Therefore, an element of~$\Omega_{n,1}^N$ is either
\begin{itemize}
\item[(i)] A scalar aroma with one marked node, times a wedge product of~$n$ distinct rooted trees; or
\item[(ii)] An unmarked scalar aroma times the wedge product of~$n-1$ distinct rooted trees times a single rooted tree with a marked node. The marked rooted tree can coincide with one of the unmarked ones.
\end{itemize}
For type (i), we first enumerate the scalar aromas with one marked node in terms of~$t(z)$ as follows. Consider the connected component containing the marked node. The marked node lies either on the cycle or on one of the trees attached to the cycle. Those in the first group are enumerated by the rooted trees with one marked node: delete the edge of the cycle that points to the marked node (the construction is invertible). Those in the second group are enumerated by sequences of {\em two} rooted trees, each with a marked node: add edges from the first root to the first marked node and from the second root to the first root, and remove the mark from the first marked node. The rooted trees with one marked node are enumerated by sequences (of any length) of rooted trees: delete the outgoing edges from the nodes on the path from the root to the marked node.

Putting this together, the rooted trees with one marked node are enumerated by the following (sequence A000107),
$$w(z):=t(z)+t(z)^2 + t(z)^3+\dots = t(z)/(1-t(z)).$$
The connected scalar aromas with one marked node are enumerated by $w(z) + w(z)^2 = t(z)/(1-t(z))^2$ (sequence A038002).
The scalar aromas with one marked node are enumerated by $a(z)t(z)/(1-t(z))^2$ (sequence A027853).
Including the unmarked scalar aroma component and the wedge product of rooted trees gives the contribution from type (i) forms as~$b_0(u,z)t(z)/(1-t(z))^2$.

For type (ii), combining the components of an unmarked form with one fewer trees (i.e. an element of~$\Omega_{n-1}$), enumerated by~$u b_0(u,z)$, and a marked rooted tree, enumerated by~$t(z)/(1-t(z))$, gives the contribution from type (ii) forms as~$u b_0(u,z)t(z)/(1-t(z))$.

Summing the results for type (i) and type (ii) forms gives the expression for~$b_1(u,z)$.

The functional forms in~$\mathcal{I}_1$ are associated with scalar aromas with one marked leaf (containing the symbol~$\Circled{1}$).
They are bijective to the scalar aromas with one marked node of degree one less: delete the marked leaf and mark the node to which it points (the process is invertible). Therefore, we deduce~$| \mathcal{I}_1^N| = |\Omega_{0,1}^{N-1}|$ and this gives the result for~$c_1(z)$.
\end{proof}

Note that
$$ \Omega_{1,1}^{N-1}\cong \Omega_{0,1}^{N-1}\cong \mathcal{I}_1^{N}$$
and that all three spaces are enumerated by the mapping patterns on~$N$ (or~$N+1$) points with one marked node.

\begin{remark}
As row 1 of the augmented bicomplex is exact, its alternating sum of dimensions is zero. This is verified directly:
\begin{align*}
\big(\sum_{n=0}^\infty (-1)^n |\Omega_{n,1}|\big) - |\mathcal{I}_1|
&= b_1(-1,z)-c_1(z) \\
& =b_0(-1,z)\frac{t(z)^2}{(1-t(z))^2} - \frac{a(z)t(z)z}{(1-t(z))^2} \\
&= \frac{a(z)z}{t(z)} \frac{t(z)^2}{(1-t(z))^2} - \frac{a(z)t(z)z}{(1-t(z))^2}\\
&= 0.
\end{align*}
\end{remark}

\begin{table}[ht]
\begin{center}
\begin{tabular}{|r|rrrrr||rrrrrr|}
\hline
& \multicolumn{5}{c||}{First row:~$|\Omega_{n}^N|$}
& \multicolumn{5}{c}{Second row:~$|\Omega_{n,1}^N|$}
&~$|\mathcal{I}_1^N|$\\
\hline
$N$ & \multicolumn{5}{c||}{$n$} & \multicolumn{5}{c}{$n$} & \\
 & 4 & 3 & 2 & 1 & 0				& 4 & 3 & 2 & 1 & 0 & \\
 \hline
 1 & 0 & 0 & 0 & 1 & 1			& 0 & 0 & 0 & 1 & 1 & 0 \\
 2 & 0 & 0 & 0 & 2 & 3 			& 0 & 0 & 1 & 4 & 4 & 1 \\
 3 & 0 & 0 & 1 & 6 & 7 			& 0 & 0 & 4 & 15 & 15 & 4\\
 4 & 0 & 0 & 3 & 16 & 19 			& 0 & 1 & 16 & 52 & 52 & 15\\
 5 & 0 & 0 & 11 & 45 & 47 			& 0 & 5 & 57 & 175 & 175 & 52 \\
 6 & 0 & 2 & 33 & 121 & 130 		& 0 & 22 & 197 & 571 & 571 & 175 \\
 7 & 0 & 7 & 102 & 338 & 343 		& 2 & 85 & 654 & 1838 & 1838 & 571\\
 8 & 0 & 29 & 298 & 929 & 951 		& 11 & 310 & 2137 & 5834 & 5834 & 1838\\
 9 & 1 & 99 & 878 & 2598 & 2615 	& 53 & 1078 & 6859 & 18363 & 18373 & 5834 \\
\hline
\end{tabular}
\caption{\label{tab:dims}Dimensions of the bottom two rows of the augmented aromatic bicomplex for orders 1 to 9.}
\end{center}
\end{table}


The following result extends Theorem~\ref{thm:divs} to the divergence-free case.
\begin{theorem}
\label{theorem:counting_div_free}
Let~$\widetilde{b}_p(u,z)$,~$\widetilde{c}_p(z)$, and~$\widetilde{s}(z)$ be the divergence-free analogues of the generating functions~$b_p(z)$,~$c_p(z)$, and~$s(z)$.
Then
\begin{align*}
\widetilde{b}_0(u,z)&= \frac{z b_0(u,z)}{t(z)},\\
\widetilde{b}_1(u,z) &= z b_0(u,z) \frac{t(z) + u - u t(z))}{(1-t(z))^2},\\
\widetilde{c}_1(z) &= z c_1(z),\\
\widetilde{s}(z) &= z+\frac{z s(z)}{t(z)} = z+\frac{z(a(z)t(z)-\mathring{a}(z))}{t(z)}.
\end{align*}
\end{theorem}

\begin{proof}
In each case we need to enumerate the non-self-looped elements of the aromatic bicomplex. The loops occur only in the scalar components. Recall that the non-self-looped scalars,~$\tOmega_{0}$, are enumerated by~$\overline{a}(z) = z a(z)/t(z)$. This gives the result for~$\widetilde{b}_0$ and, using Theorem~\ref{theorem:exact_div_free_complex},~$\widetilde{s}(z)$. 

For the second row, we first consider~$\tOmega_{0,1}^N$, the scalar aromas with~$N$ nodes, no self-loops, and one marked node indicated by~$\Circled{1}$. These are enumerated by the self-functions of~$\{1,\dots,N-1\}$ with one marked node. That is,~$\tOmega_{0,1}^N\cong \Omega_{0,1}^{N-1}$. The construction is as follows. Consider an element of~$\Omega_{0,1}^{N-1}$, that is, a directed graph
with~$N-1$ nodes, each of which has exactly one outgoing edge, with one marked node. Add a new node with an outgoing edge going to the marked node, and redirect all self-loops to the new node; then, move the mark to the new node. This gives a marked non-self-looped scalar aroma of degree~$N$. The process is invertible: given a marked scalar with no self-loops, redirect the edges that point to the marked node to themselves, mark the node pointed to by the marked node, and delete the marked node.
Therefore~$\tOmega_{0,1}^N$ has generating function~$z a(z)t(z)/(1-t(z))^2$.

For the rest of the second row,~$\tOmega_{n,1}$, recall the two types of forms, type (i) (scalar is marked) and type (ii) (tree is marked).
For type (i),  that the only change in the divergence-free case is that the scalars must be non-self-looped, as just enumerated. For type (ii), we combine the three components of non-self-looped scalars (enumerated by~$a(z)$),~$n-1$ unmarked rooted trees (enumerated by~$u b_0(u,z)/a(z)$), and one marked rooted tree (enumerated by~$t(z)/(1-t(z))$). The product of these three, plus the contribution from the forms of type (i), gives the result for~$b_1(u,z)$.

The elements of~$\widetilde{\mathcal{I}}_1^N=I(\widetilde{\Omega}_{0,p}^N)$ are linear combinations of scalar aromas with~$N$ nodes, one marked leaf, and no self-loops.
As in the general case, these are bijective to the scalar aromas with one marked node of degree one less: delete the marked leaf and mark the node to which it points. This gives the result for~$\widetilde{c}_1(z)$.
\end{proof}

Note that we now have {\em five} isomorphic spaces,
$$ \Omega_{1,1}^{N-1}\cong \Omega_{0,1}^{N-1}\cong \mathcal{I}_1^{N}\cong \widetilde{\mathcal{I}}_1^{N+1}
\cong\tOmega_{0,1}^N.$$
Each space is enumerated by the self-functions of~$\{1,\dots,N-1\}$ with one marked node (and generating
function~$z a(z)t(z)/(1-t(z))^2$), but in a different way in each case.

We remark that the second row of the divergence-free augmented aromatic bicomplex is not exact.

\subsection{Bases of the kernel and image of~$d_H$ and~$d_H^*$}
\label{section:description_image_kernel_dH}

In this subsection, we work specifically with~$d_H\colon \Omega_{1}\to \Omega_{0}$ and we describe the image and the kernel of~$d_H$ and~$d_H^*$.
We recall that~$d_H^*\colon \Omega_{0}^*\to \Omega_{1}^*$ is the dual map of~$d_H$. We described the dimension of~$\Psi=\Ker d_H$ in Subsection~\ref{section:combinatorics_Robert}. The following result describes the dimensions of~$\im d_H$,~$\Ker d_H^*$, and~$\im d_H^*$.

\begin{theorem}
The scalar divergences~$\im d_H$ have dimension~$|\mathring\Omega_{0}|$.
The conditions~$\Ker d_H^*$ that a scalar must satisfy to be a divergence have dimension~$|\overline{\Omega}_{0}|$.
The conditions~$\im d_H^*$ that a vector must satisfy to be divergence-free have dimension~$|\mathring\Omega_{0}|$.
\end{theorem}

\begin{proof}
Recall the fundamental theorem of linear algebra for a linear map~$A\colon V \to W$:
$$
|\im A| = |\im A^*| = \mathrm{rank} A, \quad
|\im A| +|\Ker A| =|V|, \quad
|\im A^*| + |\Ker A^*| =|W|.
$$
The space~$\im A^*= \mathrm{Ann}(\Ker A)$ is the annihilator of the kernel of~$A$ (i.e., the conditions that an element of~$A$ must satisfy in order to lie in the kernel) and~$\Ker A^*
 = \mathrm{Ann}(\im A)$ is the annihilator of the image of~$A$.
Choosing~$A=d_H$ gives the result.
\end{proof}

\begin{remark}
Consider the aromatic tree~$\gamma\backslash e\in \FF_{1}$ obtained by cutting one edge~$e\in E$ of~$\gamma\in \FF_{0}$,~$m_1(\gamma,e)$ the coefficient of~$\gamma$ in~$d_H(\gamma\backslash e)$, and~$m_2(\gamma,e)$ the number of edges~$\widehat{e}\in E$ of~$\gamma$ such that~$\gamma\backslash \widehat{e}=\gamma\backslash e$. Then, for~$\gamma\in \Omega_{0}$,~$d_H^*\gamma^*$ satisfies
$$d_H^*\gamma^*=\sum_{e\in E} \frac{m_1(\gamma,e)}{m_2(\gamma,e)}(\gamma\backslash e)^*.$$
\end{remark}

We now construct bases of~$\Ker d_H$,~$\im d_H$,~$\Ker d_H^*$, and~$\im d_H^*$.
We start with the basis of the solenoidal forms.
\begin{theorem}
\label{theorem:basis_solenoidal_forms}
Let~$\varphi\colon \Omega_{1}\to\mathring\Omega_{0}$ be defined by attaching a self-loop to the root, extending by linearity. Define any total order on~$\TT$, then a basis of~$\Psi=\Ker d_H$ is
$$\BB_\Psi=\{d_H(\phi \varphi(t_2)\dots\varphi(t_{l-1})\varphi(t_{l+1})\dots\varphi(t_{k}) t_1\wedge t_l),\quad
\phi\in \FF_0\cup\{\emptyset\},\quad t_1<\dots<t_k\in \TT
\}.$$
\end{theorem}

\begin{proof}
The map~$\varphi$ is surjective: given any self-looped scalar, removing one of the self-loops gives a pre-image. Therefore,~$|\Ker\varphi| = |\Ker d_H|$.
Let us first determine~$\Ker\varphi$. Consider a self-looped scalar whose distinct self-looped connected components are~$\varphi(t_1),\dots, \varphi(t_k)$ for trees~$t_1<\dots<t_k$; it can be written~$\phi\varphi(t_1)\dots\varphi(t_k)$ where~$\phi$ is a scalar. Its distinct preimages under~$\varphi$ are~$\phi \varphi(t_1) \dots t_l \dots \varphi(t_k)$ for~$l=1,\dots,k$. Therefore, the kernel of~$\varphi$ restricted to the span of these preimages has
dimension~$k-1$ and we consider
$$ \{\phi (t_1 \varphi(t_2)-\varphi(t_1)t_2)\varphi( t_3)\dots \varphi(t_k),\dots,
\phi(t_1 \varphi(t_k) - \varphi(t_1)t_k) \varphi(t_2)\dots\varphi(t_{k-1})\}$$
as a basis of~$\Ker\varphi$.
We now map~$\Ker\varphi$ to~$\Ker d_H$ by 
$$\phi_l(t_1 \varphi(t_l) - \varphi(t_1)t_l)
\mapsto
d_H(\phi_l t_1\wedge t_l),\quad \text{where} \quad \phi_l=\phi \varphi(t_2)\dots\varphi(t_{l-1})\varphi(t_{l+1})\dots\varphi(t_{k}),$$
extending by linearity. From exactness, the map is surjective. As~$|\Ker\varphi| = |\Ker d_H|$, it is an isomorphism.
\end{proof}

\begin{remark}
\label{remark:basis_div_free_solenoidal}
One could wonder whether the set~$\{d_H \wedge \gamma, \gamma\in \FF_2\}$ is a basis of~$\Ker(d_H)$. This is not the case in general.
For~$N=6$, one finds for instance the following identity
\begin{align*}
d_H&(
\wedge\atree1001 \atree2101 \atree3102
+\wedge \atree2001 \atree3102 \atree1101
+\wedge \atree3001 \atree1101 \atree2101
+\wedge \atree1101 \atree5105
+\wedge \atree3101 \atree3102
+\wedge \atree5107 \atree1101
\\&+\wedge \atree5104 \atree1101
+\wedge \atree2101 \atree4104
+\wedge \atree2101 \atree4103
+\wedge \atree2002 \atree3102 \atree1101
+\wedge \atree3002 \atree1101 \atree2101
+\wedge \atree3003 \atree1101 \atree2101
)=0.
\end{align*}
\end{remark}

The following result shows that the divergences are a graph over the self-looped scalars.
\begin{theorem}
\label{thm:divs}
For~$\alpha\in\mathring{\Omega}_{0}$, let~$k(\alpha)$ be the number of self-loops in~$\alpha$, and 
$\rho(\alpha)$ be the non-self-looped scalar obtained from~$\alpha$ as the sum of the redirection of all~$1$-loops to other nodes in all possible ways.
Then the map
$$ \mathring{\Omega}_{0} \to \im d_H,\quad \alpha \mapsto \alpha + (-1)^{k(\alpha)-1}\rho(\alpha)$$
is an isomorphism, and generates a basis of~$\im d_H$.
\end{theorem}

\begin{proof}
Let~$\mathring{V}$ be the set of nodes with self-loops of~$\alpha$.
From Proposition~\ref{prop:higher_Euler_operators_Omega}, we deduce that for~$v\in\mathring{V}$,~$\alpha-\Eulerop_v \alpha\in \Img(d_H)$.
If we have two nodes~$v,w\in\mathring{V}$, then~$\alpha-\Eulerop_v \alpha\in \Img(d_H)$ and~$\Eulerop_v \alpha-\Eulerop_w \Eulerop_v \alpha\in \Img(d_H)$, so that~$\alpha-\Eulerop_w \Eulerop_v \alpha\in \Img(d_H)$.
By applying iteratively this process, we find that
$$\alpha-\prod_{v\in\mathring{V}} \Eulerop_v \alpha=\alpha + (-1)^{k(\alpha)-1}\rho(\alpha)\in \Img(d_H).$$
As each self-looped scalar appears once in the image, the map is injective.
The map is an isomorphism as the domain and codomain have the same dimension.
\end{proof}

\begin{remark}
The operation~$\rho$ in Theorem~\ref{thm:divs} corresponds to removing all self-loops in~$\alpha$ by repeated integration by parts, as illustrated in the following example on elementary differentials:
\begin{align*}
f^i_i f^j_j &= (f^i f^j_j)_i - f^i f^j_{ij} \\
&= (f^i f^j_j)_i - (f^i f^j_i)_j+f^i_j f^j_i \\
&= (f^i f^j_j - f^j f^i_j)_i + f^i_j f^j_i.
\end{align*}
That is,~$f^i_i f^j_j-f^i_j f^j_i$ is a divergence.
We describe this comparison with integration by parts further in Subsection~\ref{section:IBP}.
\end{remark}

\begin{corollary}
\label{corollary:self_loop_is_divergence}
No non-trivial combination of non-self-looped scalars in~$\Omega_0$ is a divergence.
\end{corollary}

\begin{corollary} The conditions to be a divergence~$\Ker d_H^*$ are a graph over the dual of the non-self-looped forms in~$\Omega_0$.
\end{corollary}

In the following theorem, this graph is realized explicitly.

\begin{theorem}
Let~$\widehat{E}\subset E$ be a set of edges of the scalar aroma~$\beta\in \Omega_0$ and let~$\beta\backslash \widehat{E}$ be~$\beta$ with edges~$\widehat{E}$ replaced by self-loops. Let~$m(\beta,\widehat{E})$ be the number of ways that redirecting self-loops of~$\beta\backslash \widehat{E}$ results in~$\beta$.
Let~$\pi\colon\Omega_{0}\to\Omega_{0}$ be defined by
$$ \pi(\beta)=\sum_{\widehat{E}\subset E}(-1)^{|\widehat{E}|} m(\beta,\widehat{E}) \beta\backslash \widehat{E}$$
Then
$$ \{\pi(\beta), \beta\in\overline{\Omega}_{0}^*\}$$
is a basis of~$\mathrm{Ann}(\im A)$, the conditions that a scalar must satisfy to be a divergence.
\end{theorem}
\begin{proof}
The construction is directly related to that in Theorem~\ref{thm:divs}. The conditions for~$(x,y)$ to lie on the graph~$\{x,Ax)\colon x\in \mathbb{R}^n\}\subset \R^n\times\R^m$ are~$y-Ax=0$,~$A\in\R^{m\times n}$. A basis for these conditions is given by the rows of~$y-Ax$, where Theorem
\ref{thm:divs} gives the columns of~$A$. That is, for each non-self-looped scalar~$\beta$ we need to determine the coefficient of~$\beta$ in~$\rho(\alpha)$ for each self-looped scalar~$\alpha$. This is the expression for~$\pi$: the term~$\widehat{E}=\emptyset$ gives~$\beta$, and the terms from non-empty sets of edges~$\widehat{E}$ give the~$\alpha$'s that can give rise to~$\beta$.
\end{proof}

Finally, we present a basis of~$\im d_H^*$. It is quite straightforward, as we can find a suitable subspace on which~$d_H^*$ is injective.

\begin{theorem}
The set
$$\{d_H^*\phi^*, \phi\in\mathring{\Omega}_{0}\}$$
is a basis of~$\im d_H^*$, the conditions that a form in~$\Omega_{0}$ must satisfy to be divergence free.
\end{theorem}
\begin{proof}

For any linear map~$A\colon V\to W$,~$\langle A^*w^*,v\rangle = \langle w^*,A v\rangle = 0$ for all~$w^*\in W^*$ when~$v\in\Ker A$. 
Thus, divergence-free vectors do satisfy the given conditions. Furthermore, the dimension of the set is correct. It remains to show that the set is linearly independent. This is the same as showing that~$d_H^*|_{\mathring{\Omega}_{0}^*}$ is injective.

Let~$\mathrm{pr}\colon\Omega_{0}\to\mathring\Omega_{0}$ be the natural projection to the self-looped scalars. From Theorem~\ref{thm:divs}, the divergences form a graph over the self-looped scalars. That is,~$\mathrm{pr}\circ d_H$ is surjective. Therefore its dual~$d_H^*|_{\mathring{\Omega}_{0}^*}$ is injective.
\end{proof}

\subsection{Integration by parts of aromatic forests}
\label{section:IBP}

The horizontal homotopy operator is often described in the differential geometry literature as an integration by parts operator.
The concept of integration by parts of trees was also introduced in the context of stochastic numerical analysis in~\cite{Laurent20eab,Laurent21ocf} on exotic aromatic B-series (see also~\cite{Bronasco22ebs}). We show in this section that a similar integration by parts process can be adapted in the context of aromatic forms to define a different horizontal homotopy operator on~$\Omega_{0,p}$.

Let~$\gamma\in\Omega_{0,p}^N$ a linear combination of forests, let~$\tau\in \FF_{0,p}^N$ one of these forests and~$v$ a vertex of~$\tau$ on a 1-loop. We denote~$a(\tau)$ the coefficient of~$\tau$ in~$\gamma$, and~$\theta_v(\tau)$ the forest~$\tau$ where we remove the edge linking~$v$ to itself and transform~$v$ into a root.
The alternative horizontal homotopy operator~$\widehat{h}_H$ on~$\Omega_{0,p}$ is given by the following algorithm on~$\Omega_{0,p}^N$, and extended to~$\Omega_{0,p}$ by linearity.
\begin{algorithm}[H]
\renewcommand{\thealgorithm}{Homotopy operator~$\widehat{h}_H$}
\caption{}
\begin{algorithmic}
\STATE Given~$\gamma\in\Omega_{0,p}^N$, initialize~$\widehat{h}_H\gamma=0$ and~$\widehat{\gamma}=\gamma-\abs{\gamma}^{-1}\Eulerop \gamma$.
\WHILE{there is a 1-loop in a forest~$\tau$ of~$\widehat{\gamma}$ on a vertex~$v$}
	\STATE~$\widehat{h}_H\gamma\leftarrow \widehat{h}_H\gamma+a(\tau) \theta_v(\tau)$,
	\STATE~$\widehat{\gamma}\leftarrow \widehat{\gamma} - a(\tau) d_H \theta_v(\tau)$.
\ENDWHILE
\RETURN~$\widehat{h}_H\gamma$
\end{algorithmic}
\end{algorithm}

Note that each iteration in the algorithm reduces the number of 1-loops by one. Thus, the algorithm always terminates.
We emphasize that the result of the algorithm is independent of the order in which we detach the~$1$-loops. This is not the case in the similar algorithm proposed in~\cite{Laurent20eab}, as there is an extra term involved in the integration by parts process.

\begin{theorem}
For~$\gamma\in\Omega_{0,p}$, the output~$\widehat{h}_H\gamma$ of the algorithm is the horizontal homotopy operator, up to a divergence-free term, that is,
$$d_H(h_H\gamma - \widehat{h}_H\gamma)=0.$$
\end{theorem}

\begin{proof}
After the algorithm terminates,~$\widehat{\gamma}$ is given by
$$
\widehat{\gamma}=\gamma-\frac{1}{\abs{\gamma}}\Eulerop \gamma - d_H \widehat{h}_H \gamma=d_H(h_H\gamma -\widehat{h}_H \gamma),
$$
where~$\widehat{\gamma}$ does not contain any 1-loop and where we used Theorem~\ref{theorem:exact_variational_chain}.
We deduce from Proposition~\ref{prop:higher_Euler_operators_Omega} that~$\widehat{\gamma}\in \Img(d_H)$. According to Corollary~\ref{corollary:self_loop_is_divergence}, we find~$\widehat{\gamma}=0$.
\end{proof}

We now have two different ways to compute the horizontal homotopy operator on~$\Omega_{0,p}$. The first one, presented in Subsection~\ref{section:Euler_operators}, uses the Euler operators. The second one, in the spirit of~\cite{Laurent20eab}, is based on the repeated use of detaching and grafting operations on specific nodes.
We emphasize that the expressions of the homotopy operator given by these two methods are different in general, but are always equal up to a divergence-free term.
The two homotopy operators can produce both concise and tedious outputs, and they outperform each other in this manner on different forests. We refer the reader to Table~\ref{table:comparison_homotopy} for some examples. This difference in the number of terms increases rapidly with the order; for instance, for the form~$\gamma=\atree1001\atree1001\atree3101$,~$h_H \gamma$ has 6 terms, while~$\widehat{h}_H \gamma$ has 26 terms. On the other hand, it is possible to find examples where~$\widehat{h}_H~$ produces less terms than~$h_H$.
It would be interesting to find a homotopy operator with a minimal number of terms in the output, or a procedure to simplify the outputs of a homotopy operator in the spirit of \cite[Sect.\ts IV.B]{Anderson89tvb}.
Moreover, it is not known whether a similar approach in the divergence-free context could yield a different homotopy operator. This is matter for future work.

\begin{table}[!htb]
	\setcellgapes{3pt}
	\centering
	\begin{tabular}{|c|c|c|}
	\hline
	$\gamma\in\Omega_0$ &~$h_H\gamma$ &~$\widehat{h}_H\gamma$ \\
	\hhline{|=|=|=|}
	$\atree1001$ &~$\atree1101$ &~$\atree1101$ \\
	\hhline{|=|=|=|}
	$\atree2001$ &~$0$ &~$0$ \\
	\hline
	$\atree2002$ &~$\atree2101$ &~$\atree2101$ \\
	\hline
	$\atree1001\atree1001$ &~$\atree1001\atree1101$ &~$\atree1001\atree1101$ \\
	\hhline{|=|=|=|}
	$\atree3001$ &~$\frac{1}{6}\atree2001\atree1101+\frac{1}{6}\atree1001\atree2101-\frac{1}{6}\atree3101-\frac{1}{6}\atree2002\atree1101~$ &~$\frac{1}{3}\atree1001\atree2101- \frac{1}{3}\atree2002\atree1101$ \\
	\hline
	$\atree3002$ &~$\frac{1}{6}\atree3101+\frac{1}{6}\atree2002\atree1101-\frac{1}{6}\atree2001\atree1101-\frac{1}{6}\atree1001\atree2101$ &~$\frac{1}{3}\atree2002\atree1101- \frac{1}{3}\atree1001\atree2101$ \\
	\hline
	$\atree3003$ &~$\atree3102$ &~$\atree3102$ \\
	\hline
	$\atree3004$ &~$\frac{2}{3}\atree2001\atree1101 +\frac{1}{3}\atree3101$ &~$\atree3101+\frac{2}{3}\atree1001\atree2101-\frac{2}{3}\atree2002\atree1101$ \\
	\hline
	$\atree2001 \atree1001$ &~$0$ &~$0$ \\
	\hline
	$\atree2002 \atree1001$ &~$\frac{1}{3}\atree2002\atree1101+\frac{2}{3}\atree1001\atree2101$ &~$\atree2002\atree1101+\frac{2}{3}\atree2001\atree1101-\frac{2}{3}\atree3101~$ \\
	\hline
	$\atree1001 \atree1001 \atree1001$ &~$\atree1001 \atree1001 \atree1101$ &~$\atree1001 \atree1001 \atree1101$ \\
	\hline
	\end{tabular}
	\caption{Comparison of the horizontal homotopy operators on~$\Omega_0$ for the first orders.}
	\label{table:comparison_homotopy}
	\setcellgapes{1pt}
\end{table}


\subsection{Explicit description of volume-preserving aromatic integrators}
\label{section:application_vp_integrators}

It is known that the only volume-preserving consistent B-series method is the exact flow~\cite{Chartier07pfi,Iserles07bsm}.
In~\cite{MuntheKaas16abs}, the question of the existence of a volume-preserving aromatic method is raised, where an aromatic method is a one-step integrator that has an expansion as an aromatic B-series.
In~\cite{Bogfjellmo19aso}, a methodology to create pseudo-volume-preserving interators is proposed, by substituting in a standard Runge-Kutta method the vector field~$f$ with an aromatic B-series.
We give in this subsection an explicit expression of the form of a volume-preserving aromatic method, and we use it to prove that there does not exist any aromatic Runge-Kutta integrator and to discuss the form of a volume-preserving aromatic B-series method.

Consider a consistent one-step integrator~\eqref{equation:one-step_integrator} for solving the differential equation~\eqref{equation:ODE} with the assumption~$\Div(f)=0$.
We assume the integrator~\eqref{equation:one-step_integrator} has an expansion in aromatic B-series given by the linear coefficient map~$a\colon \widetilde{\Omega}_1 \rightarrow \R$; that is, its (formal) Taylor expansion has the form
$$
\Phi(y,h)=y+F(B(a))(hf),
\quad
B(a)=\sum_{\tau\in\FF_1} \frac{a(\tau)}{\sigma(\tau)}\tau,$$
where~$\sigma(\tau)$ is the cardinal of the set of automorphisms on the set of nodes~$V$ of~$\tau$ that leave~$\tau$ unchanged (see~\cite{Bogfjellmo19aso}).
We call such a method an aromatic B-series method. If in addition the integrator reduces to a standard Runge-Kutta method when choosing a vector field~$f$ that satisfies~$F(\gamma)(f)=0$ for all~$\gamma\in \FF_1\setminus \TT$, we call the integrator an aromatic Runge-Kutta method.

An aromatic B-series method can be seen as the exact solution of the modified ODE~\eqref{equation:modified_ODE}, and the modified flow is given by the aromatic B-series~$B(b)$ satisfying~$B(a)=B(b)\triangleright B(e)$. The operation~$\triangleright$ is the substitution of B-series  and~$b\colon \widetilde{\Omega}_1\rightarrow \R$ is the coefficient map of the modified flow.
It is known that~$a$ and~$b$ satisfy~$b \star e=a$, where~$\star$ is the substitution of B-series coefficients~\cite{Calaque11tih,Chartier10aso,Bogfjellmo19aso}.
The map~$e$ is the coefficient of the exact flow of~\eqref{equation:ODE}. Its expression is given for instance in \cite[Chap.\ts III]{Hairer06gni} for Butcher trees. It is extended to the aromatic trees by~$e(\gamma)=0$ if~$\gamma$ is composed of at least an aroma.
The question raised in~\cite{MuntheKaas16abs} is the following: can we find an aromatic B-series method such that
$d_H B(b)=0$, that is, such that the modified B-series~$B(b)$ is solenoidal.
Note that choosing~$a=e$ yields a simple solution to the problem, but there does not exist any reasonable numerical method whose coefficient map is given by the exact flow coefficient~$e$.

The main motivation for considering aromatic B-series instead of standard B-series comes from the following negative result.
\begin{theorem}[\cite{Chartier07pfi,Iserles07bsm}]
\label{theorem:Iserles_theorem}
The solenoidal combinations of rooted trees satisfy
$$\Span(\TT)\cap \Psi=\emptyset, \quad \Span(\TT)\cap \widetilde{\Psi}=\Span(\atree1101).$$
In particular, the only volume-preserving consistent B-series method is the exact flow.
\end{theorem}

Note that in the standard context, Theorem~\ref{theorem:Iserles_theorem} is a consequence of Theorem~\ref{thm:divs}.
Indeed, let~$v=\sum_i c_i t_i$,~$t_i\in\mathcal{T}$, be a combination of rooted trees. As the divergences are graphs over the self-looped scalars (see Theorem~\ref{thm:divs}),~$d_H v=0$ if the coefficient of each self-looped scalar in~$d_H v$ is zero. But~$\mathrm{pr} d_H v=\sum_i c_i \theta(t_i)$, where~$\mathrm{pr}\colon\Omega_{0}\to\mathring\Omega_{0}$ is the natural projection to the self-looped scalars, and the self-looped scalars~$\theta(t_i)$ are linearly independent.
%
%

Following Theorem~\ref{theorem:Iserles_theorem}, one is interested in finding a class of non-trivial volume-preserving consistent aromatic B-series methods.
We deduce from the previous discussion and Theorems \ref{theorem:exact_aro_var_bi}, \ref{theorem:exact_div_free_complex}, and \ref{theorem:Kernel_divergence_free} the following explicit description of the coefficients of an aromatic volume-preserving integrator.
\begin{theorem}
If~$B(a)$ is the aromatic B-series of a consistent volume-preserving integrator, then there exists~$\eta\in \widetilde{\Omega}_2$ such that the modified flow is a B-series of the form~$B(\atree1101+ d_H \eta)$ and~$B(a)$ is given by the substitution
$$B(a)=(\atree1101+ d\eta)\triangleright B(e).$$
More precisely, there exists a coefficient map~$\alpha\colon \widetilde{\Omega}_2\rightarrow \R$ such that~$a$ is given by
\begin{equation}
\label{equation:explicit_vp_B-series}
a=(\atree1101^*+\alpha d_H^* A_\sigma)\star e,\quad \text{where} \quad A_\sigma\gamma=\frac{1}{\sigma(\gamma)}\gamma.
\end{equation}
\end{theorem}


For the first orders, the B-series of a volume-preserving aromatic B-series method has the form:
\begin{align*}
B(a)&=\atree1101
+\frac{1}{2}\atree2101
+\frac{1}{6}\atree3102+\Big(\frac{1}{6}-\frac{1}{2}\alpha(\wedge\atree1101 \atree2101)\Big)\atree3101+\frac{1}{2}\alpha(\wedge\atree1101 \atree2101)\atree2002 \atree1101\\
&+
\frac{1}{24}\atree4102
+\Big(\frac{1}{8}-\frac{1}{2}\alpha(\wedge\atree1101 \atree3102)-\alpha(\wedge\atree1101 \atree3101)-\frac{1}{2}\alpha(\wedge\atree1101 \atree2101)\Big)\atree4104\\&
+\Big(\frac{1}{24}-\frac{1}{2}\alpha(\wedge\atree1101 \atree3102)+\frac{1}{2}\alpha(\wedge\atree1101 \atree3101)-\frac{1}{4}\alpha(\wedge\atree1101 \atree2101)\Big)\atree4103
+\frac{1}{2}\alpha(\wedge\atree1101 \atree3102)\atree3003 \atree1101
+\frac{1}{2}\alpha(\wedge\atree1101 \atree2101)\atree2002 \atree2101\\&
+\Big(\frac{1}{24}-\frac{1}{2}\alpha(\wedge\atree1101 \atree3101)-\frac{1}{4}\alpha(\wedge\atree1101 \atree2101)\Big)\atree4101
+\Big(\frac{1}{2}\alpha(\wedge\atree1101 \atree3102)+\alpha(\wedge\atree1101 \atree3101)+\frac{1}{2}\alpha(\wedge\atree1101 \atree2101)\Big)\atree3002 \atree1101
+\dots
\end{align*}
Note that the coefficients of the bamboo trees (or tall trees)~$\BB\TT=\{\atree1101,\atree2101,\atree3102,\dots\}$ coincide with the ones of the exact flow. This fact has been noticed for standard B-series in particular in~\cite{Kang95vpa} (see also \cite[Lemma IV.3.2]{Hairer06gni}).
We deduce from this observation the following result.
\begin{theorem}
\label{theorem:no_aromatic_RK}
An aromatic Runge-Kutta method cannot be volume-preserving.
\end{theorem}

\begin{proof}
The only bamboo tree that appears in solenoidal forms is~$\atree1101$.
Indeed~$d_H^*$ vanishes on~$\BB\TT$ as the only solenoidal forms where a bamboo tree can appear are of the form~$d_H\wedge \tau_1 \tau_2$ where~$\tau_1$,~$\tau_2\in \BB\TT$, and no bamboo tree appears in these forms.
According to the identity~\eqref{equation:explicit_vp_B-series}, the B-series of an aromatic volume-preserving method has to coincide with the exact flow~$B(e)$ on the bamboo trees~$\BB\TT$.
A Runge-Kutta integrator cannot be exact on all bamboo trees. As the aromatic forests represent different elementary differentials (see~\cite{Iserles07bsm}), any aromatic integrator that reduces to a standard Runge-Kutta integrator when sending the aromas to zero cannot be volume-preserving.
\end{proof}

The methodology proposed in \cite[Sect.\ts 7]{Bogfjellmo19aso} to obtain volume-preservation of high order and the approach in \cite[Sect.\ts 9]{MuntheKaas16abs} give classes of aromatic integrators that can preserve volume up to a high-order, but that cannot be volume-preserving.
To build an aromatic volume-preserving method, it is fundamental to start with an ansatz that is exact on bamboo trees (that is, the method is exact for linear problems). A natural guess is to consider aromatic integrators that reduce to exponential Rosenbrock integrators (see, for instance,~\cite{Berland05bsa,Hochbruck05eer,Luan13ebs}) when sending the aromas to zero.
This calls for future works that study the substitution law and the variational bicomplex directly on aromatic exponential B-series, in order to find a volume-preserving aromatic B-series method.

\section{Conclusion and future work}

In this work, we introduced a new algebraic object, called the aromatic bicomplex, for the study of aromatic forms. We studied the exactness of the bicomplex in the standard case and in the divergence-free case. To this end, we introduced the Euler operators, the homotopy operators, as well as an augmented bicomplex.
The algebraic properties we proved have concrete consequences on the numerical analysis of volume-preserving integrators. They allow to describe completely the solenoidal forms and the B-series of an aromatic volume-preserving method. In particular, we proved that there are no volume-preserving aromatic Runge-Kutta methods.

Many theoretical and applied questions arise from the present work.
Following the results of Subsection~\ref{section:application_vp_integrators}, it would be interesting to rewrite the substitution and divergence operations in the context of exponential B-series, in order to find an aromatic exponential volume-preserving method.

The integration by parts of (exotic) aromatic forests is a new operation that has applications in stochastic numerical analysis and in the study of volume-preserving integrators. To the best of our knowledge, few works study the structure of (exotic) aromatic forests equipped with the integration by parts process. In particular, there is no explicit expression for the output of the integration by parts process in the stochastic setting~\cite{Laurent20eab}. An exact formula would greatly benefit the creation of high-order methods for solving ergodic SDEs.

There is a considerable literature on the variational bicomplex and the De Rham cohomology (see~\cite{Anderson89tvb} and references therein). It would be interesting to generalise some of the existing results in the context of aromatic forms.
For instance, one could try to find simpler expressions for the homotopy operators (see Subsection~\ref{section:IBP}), to find an augmented bicomplex in the divergence-free case (see Subsection~\ref{section:augmented_bicomplex}).
Two major applications of the variational bicomplex are the Noether's theorems and the study of the Laplace-De Rham operator
$\Delta=d_H d_H^* +d_H^* d_H.$
It would be interesting to see how these results translate to aromatic forms.

%
%
%
%
%
%
%
%
%
%
%
%

\bigskip

\noindent \textbf{Acknowledgements.}\
The work of Adrien Laurent was supported by the Research Council of Norway through project 302831 “Computational Dynamics and Stochastics on Manifolds” (CODYSMA).
Robert McLachlan acknowledges the support of the Simons Foundation and the hospitality of the Isaac Newton Institute for the Mathematical Sciences through their program ``Geometry, compatibility and structure preservation in computational differential equations'', where part of this work was conducted.

\bibliographystyle{abbrv}
\bibliography{Ma_Bibliographie}

\vskip-1ex
\begin{appendices}

\section{First solenoidal forms}
\label{section:solenoidal_forms}

We write the generators of the solenoidal forms~$\widetilde{\Psi}^N$ in the divergence-free case for the first orders in Table \ref{figure:first_generators}. As a consequence of Theorem~\ref{theorem:Kernel_divergence_free}, we find all the generators by computing~$d_H\gamma$ for~$\gamma\in \widetilde{\FF}_2^N$, and by adding the trivial tree~$\atree1101$. For $N\leq 5$, we observe that they form a basis of the solenoidal forms (see Remark~\ref{remark:basis_div_free_solenoidal}).
Note how no bamboo trees appear in the solenoidal forms, as discussed in Subsection~\ref{section:application_vp_integrators}.

\begin{table}[!htb]
	\setcellgapes{3pt}
	\centering
	\begin{tabular}{|c|c|l|}
	\hline
	$N$ &~$\gamma\in \widetilde{\FF}_2^N$ & \multicolumn{1}{c|}{Solenoidal form~$2d_H\wedge\gamma$} \\
	\hhline{|=|=|=|}
	$1$ &  &~$\atree1101$ \\
	\hline
	$3$ &~$\atree1101 \atree2101$ &~$\atree2002 \atree1101-\atree3101$ \\
	\hline
	$4$ &~$\atree1101 \atree3102$ &~$\atree3002 \atree1101+\atree3003 \atree1101-\atree4103-\atree4104$ \\
	\cline{2-3}
	 &~$\atree1101 \atree3101$ &~$2\atree3002 \atree1101+\atree4103-2\atree4104-\atree4101$\\
	\hline
	$5$ &~$\atree1101 \atree4102$ &~$\atree4002 \atree1101 +\atree4003 \atree1101 +\atree4004 \atree1101 -\atree5102-\atree5104-\atree5105$ \\
	\cline{2-3}
	 &~$\atree1101 \atree4103$ &~$\atree4005 \atree1101+2\atree4003 \atree1101+\atree5102-2\atree5104-\atree5103-\atree5106$\\
	\cline{2-3}
	 &~$\atree1101 \atree4104$ &~$\atree4002 \atree1101+\atree4006 \atree1101+\atree4003 \atree1101+\atree5104-\atree5105-\atree5106-\atree5107-\atree5108~$\\
	\cline{2-3}
	 &~$\atree1101 \atree4101$ &~$3\atree4005 \atree1101 +\atree5103-3\atree5108-\atree5101$\\
	\cline{2-3}
	 &~$\atree2101 \atree3102$ &~$\atree3002 \atree2101+\atree3003 \atree2101+\atree5105-\atree2002 \atree3102-\atree5104-\atree5107$\\
	\cline{2-3}
	 &~$\atree2101 \atree3101$ &~$2\atree3002 \atree2101+\atree5102+\atree5106-\atree2002 \atree3101-2\atree5105-\atree5108$\\
	\cline{2-3}
	 &~$\atree2002 \atree1101 \atree2101$ &~$\atree2002 \atree2002 \atree1101+2 \atree4002 \atree1101-\atree2002 \atree3101 -2 \atree3002 \atree2101$\\
	\hline
	\end{tabular}
	\caption{Generators of the solenoidal forms~$\widetilde{\Psi}^N$ for the first orders~$N$.}
	\label{figure:first_generators}
	\setcellgapes{1pt}
\end{table}

\section{The aromatic bicomplex for the first orders}
\label{section:examples_aromatic_bicomplex}

We present in Figures \ref{figure:ex_bicomplex_1_2} and \ref{figure:ex_bicomplex_3} the augmented aromatic bicomplex for~$N=1$,~$2$,~$3$ in the standard case, as defined in Subsection~\ref{section:intro_bicomplex}. The divergence-free aromatic bicomplex is deduced from it by deleting the~$1$-loops, and the extra column on the right.
We give a basis of each space, and we omit for conciseness the trivial spaces surrounding the bicomplex, and the wedge~$\wedge$ when writing the aromatic forms in the diagrams.
Note that the alternate sum of dimensions in each horizontal and vertical sequence, and in the Euler-Lagrange complex adds up to zero, as a consequence of Theorem~\ref{theorem:exact_aro_var_bi} and Theorem~\ref{theorem:exact_generalized_bicomplex}.

\begin{figure}[!ht]
$$\begin{tikzcd}
\atree1111 \arrow{r}{d_H} & \atree1011 \arrow{r}{I} & 0\\
\atree1101 \arrow{r}{d_H} \arrow{u}{d_V} & \atree1001 \arrow{u}{d_V} \arrow{ru}{\delta_V} & 
\end{tikzcd}
\qquad
\begin{tikzcd}
\atree1111 \atree1121 \arrow{r}{d_H} & \atree2123,\atree1011 \atree1121 \arrow{r}{d_H} & \atree2021 \arrow{r}{I} & 0 \\
\atree1101 \atree1111 \arrow{r}{d_H} \arrow{u}{d_V} & \atree2111,\atree2112,\atree1001 \atree1111,\atree1011 \atree1101 \arrow{r}{d_H} \arrow{u}{d_V} & \atree2012,\atree2013,\atree2011,\atree1001 \atree1011 \arrow{r}{I} \arrow{u}{d_V} & \atree2013 \arrow{u}{\delta_V} \\
0 \arrow{r}{d_H} \arrow{u}{d_V} & \atree2101,\atree1001 \atree1101 \arrow{r}{d_H} \arrow{u}{d_V} & \atree2001,\atree2002,\atree1001\atree1001 \arrow{u}{d_V} \arrow{ru}{\delta_V} & 
\end{tikzcd}$$
\caption{The augmented aromatic bicomplex for~$N=1$ and~$N=2$. The wedges are omitted for conciseness.}
\label{figure:ex_bicomplex_1_2}
\end{figure}
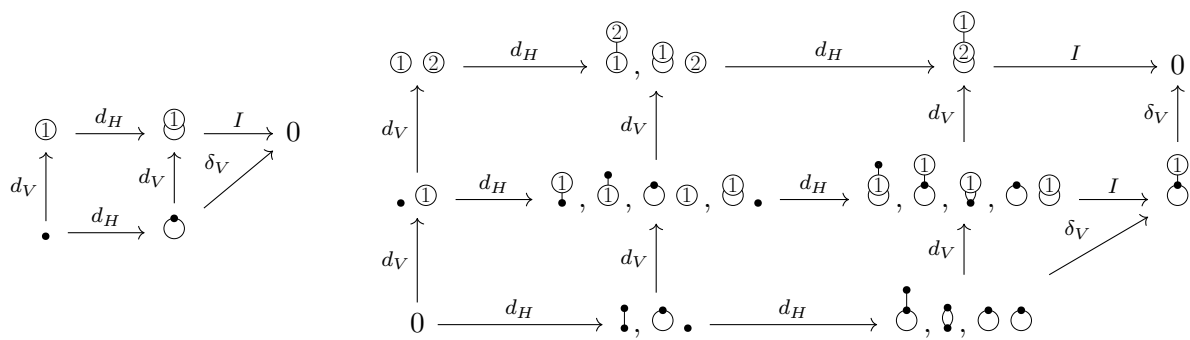


\begin{figure}[!ht]
\begin{center}
\rotatebox{90}{
$$\begin{tikzcd}[ampersand replacement=\&]
\atree1111 \atree1121 \atree1131 \arrow{r}{d_H} \& \atree1111 \atree2131,\atree1011 \atree1121 \atree1131 \arrow{r}{d_H} \& \atree3131,\atree1011 \atree2131,\atree2021 \atree1131 \arrow{r}{d_H} \& \atree3032,\atree3033,\atree3031,\atree2021 \atree1031 \arrow{r}{I} \& 
\begin{array}{l}
\atree3032+\atree3033,2\atree3032\\+\atree2021 \atree1031
\end{array} \\
\atree1101 \atree1111 \atree1121 \arrow{r}{d_H} \arrow{u}{d_V} \& 
\begin{array}{l}
\atree1101 \atree2123,\atree1111 \atree2121,\atree1111 \atree2124,\\\atree1001 \atree1111 \atree1121,\atree1011 \atree1101 \atree1121
\end{array}
 \arrow{r}{d_H} \arrow{u}{d_V} \& 
 \begin{array}{l}
\atree3124,\atree3123,\atree3122,\atree3125,\\\atree1001 \atree2123,\atree1011 \atree2121,\atree1011 \atree2124,\\\atree2021 \atree1101,\atree2013 \atree1121,\atree2012 \atree1121,\\\atree2011 \atree1121,\atree1001 \atree1011 \atree1121
\end{array}
  \arrow{r}{d_H} \arrow{u}{d_V} \& 
\begin{array}{l}
\atree3021,\atree3022,\atree3023,\atree3024,\atree3025,\atree3026,\\\atree3027,\atree3028,\atree2021 \atree1001,\\\atree2013 \atree1021,\atree2012 \atree1021,\atree2011 \atree1021
\end{array}
  \arrow{r}{I} \arrow{u}{d_V} \& 
\begin{array}{l}
\atree3021+2\atree3028-\atree2021 \atree1001\\-\atree3024,
\atree3021-\atree3022+\atree3024,\\
\atree3025,\atree3022-\atree2013 \atree1021
\end{array}
 \arrow{u}{\delta_V} \\
0 \arrow{r}{d_H} \arrow{u}{d_V} \& 
\begin{array}{l}
\atree1101\atree2111,\atree1101\atree2112,\\\atree1111\atree2101,\atree1001 \atree1101\atree1111
\end{array}
 \arrow{r}{d_H} \arrow{u}{d_V} \& 
\begin{array}{l}
\atree3111,\atree3113,\atree3114,\atree3115,\\\atree3112,\atree1011 \atree2101,\atree1001 \atree2112,\atree1001 \atree2111,\\\atree2012 \atree1101,\atree2013 \atree1101,\atree2001 \atree1111,\atree2011 \atree1101,\\\atree2002 \atree1111,\atree1001 \atree1011 \atree1101,\atree1001 \atree1001 \atree1111
\end{array}
 \arrow{r}{d_H} \arrow{u}{d_V} \& 
\begin{array}{l}
\atree3019,\atree3018,\atree3017,\atree3016,\atree3014,\atree3015,\\\atree3013,\atree3011,\atree3012,\atree2001 \atree1011,\\\atree2012 \atree1001,\atree2013 \atree1001,\atree2002 \atree1011,\\\atree2011 \atree1001,\atree1001 \atree1001 \atree1011
\end{array}
  \arrow{r}{I} \arrow{u}{d_V} \& 
\begin{array}{l}
  \atree3019,\atree3016,\\\atree3012,\atree2013 \atree1001
\end{array}
  \arrow{u}{\delta_V} \\
0 \arrow{r}{d_H} \arrow{u}{d_V} \& \atree1101 \atree2101 \arrow{r}{d_H} \arrow{u}{d_V} \& 
\begin{array}{l}
\atree3102,\atree3101,\atree1001 \atree2101,\atree2001 \atree1101,\\\atree2002 \atree1101,\atree1001 \atree1001 \atree1101
\end{array}
 \arrow{r}{d_H} \arrow{u}{d_V} \& 
\begin{array}{l}
\atree3001,\atree3002,\atree3003,\atree3004,\\\atree2001 \atree1001,\atree2002 \atree1001,\atree1001 \atree1001 \atree1001
\end{array}
 \arrow{u}{d_V} \arrow{ru}{\delta_V} \& 
\end{tikzcd}$$
}
\end{center}
\caption{The augmented aromatic bicomplex for~$N=3$. The wedges are omitted for conciseness.}
\label{figure:ex_bicomplex_3}
\end{figure}
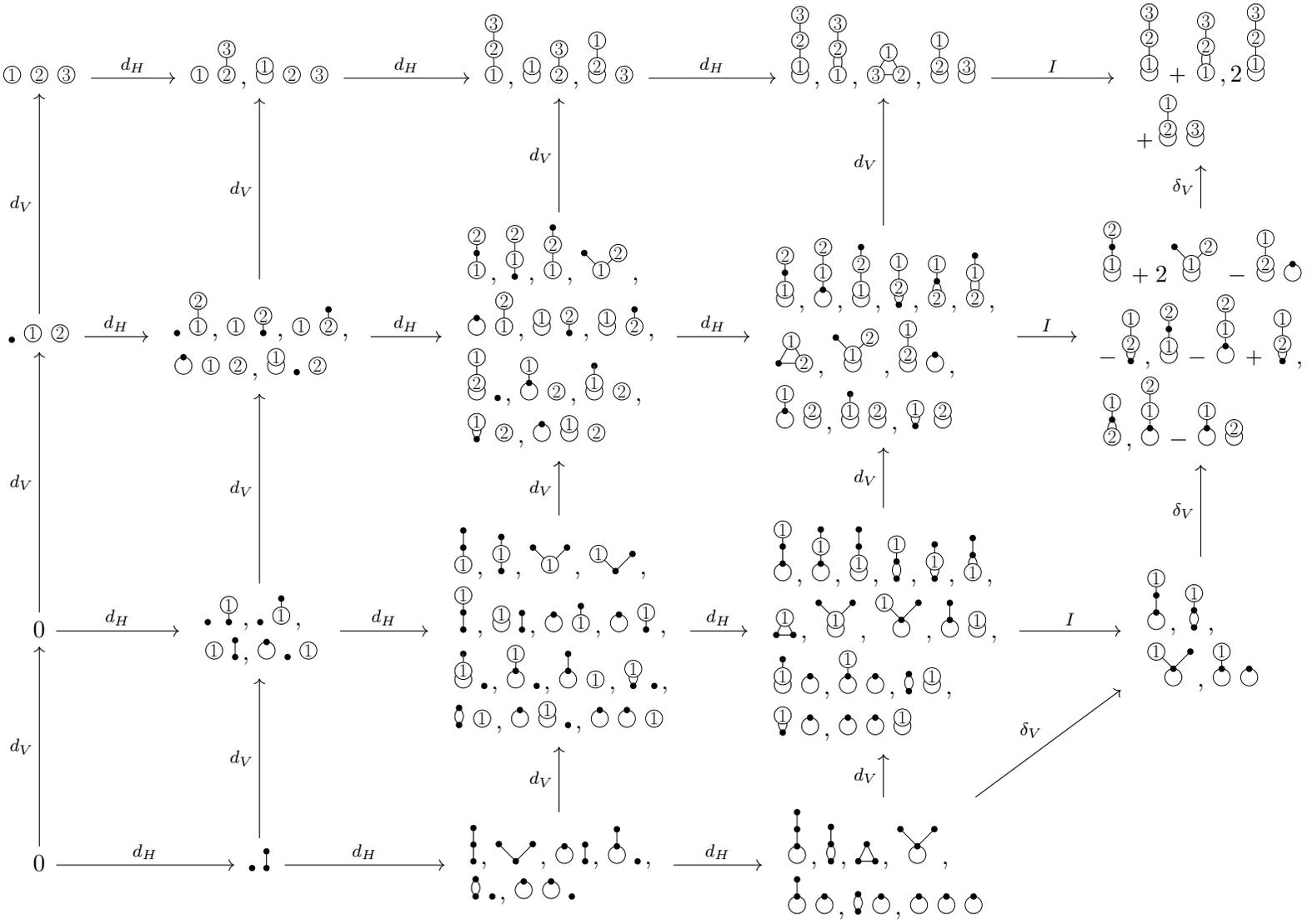

\end{appendices}

\end{document}